\documentclass[11pt]{amsart}
\usepackage{amsmath}
\usepackage{amsfonts}
\usepackage{latexsym}
\usepackage{amssymb, color}
\evensidemargin=\oddsidemargin \textwidth=160mm
\chardef\bslash=`\\ 

\makeatletter
\def\verbatim{\interlinepenalty\@M \@verbatim
  \leftskip\@totalleftmargin\advance\leftskip2pc
  \frenchspacing\@vobeyspaces \@xverbatim}
\makeatother
\newtheorem{thm}{Theorem}[section]

\newtheorem{lem}[thm]{Lemma}
\newtheorem{prop}[thm]{Proposition}
\newtheorem{ex}[thm]{Example}
\newtheorem{rem}[thm]{Remark} 

\numberwithin{equation}{section}

\newcommand{\ZZ}{{\mathbb Z}}

\begin{document}
\bibliographystyle{plain}

\title[Polynomial control of lower stability bounds and inversion norms] {Polynomial control on
 weighted  stability bounds and inversion norms of localized matrices on simple graphs}

\author{Qiquan Fang, Chang Eon Shin
 and  Qiyu Sun}

  \thanks{
  The project is partially supported by    NSF of China (Grant Nos.11701513, 11771399, 11571306), the Basic
Science Research Program through the National Research Foundation of Korea (NRF) funded by the Ministry of
Education, Science and Technology (NRF-2019R1F1A1051712) and
the National Science Foundation (DMS-1816313)
}

\address{Qiquan Fang: Department of Mathematics, Zhejiang University of Science and Technology, Hangzhou, Zhejiang, 310023, China. Email: qiquanfang@163.com}

\address{Chang Eon Shin:  Department of Mathematics, Sogang University, Seoul, 04109,  Korea.
 Email: shinc@sogang.ac.kr }

\address{Qiyu Sun: Department of Mathematics,  University of Central Florida,
Orlando, FL 32816, USA. Email: qiyu.sun@ucf.edu}


\date{\today }




\maketitle
\begin{abstract}
 The (un)weighted stability for some matrices is one of essential hypotheses in time-frequency analysis and
  applied harmonic analysis. In the first part of this paper, we show that
  for a localized matrix in a Beurling algebra, its weighted stabilities for different exponents and Muckenhoupt weights are  equivalent to each other,
  and   reciprocal of its optimal lower stability bound for one exponent and  weight is controlled by a polynomial of
  reciprocal
 of  its optimal lower stability bound for another exponent and weight.
 Inverse-closed Banach subalgebras of
matrices with certain off-diagonal decay can be informally interpreted as localization preservation
under inversion, which is of great importance in many mathematical and engineering fields.
Let ${\mathcal B}(\ell^p_w)$ be the Banach algebra  of bounded operators on the weighted sequence space $\ell^p_w$ on a simple graph.
In the second part of this paper, we prove that  Beurling algebras of  localized matrices on a simple graph are inverse-closed in  ${\mathcal B}(\ell^p_w)$
for all $1\le p<\infty$ and  Muckenhoupt $A_p$-weights $w$, and  the Beurling  norm of  the  inversion of a matrix  $A$ is bounded by a bivariate polynomial of the Beurling  norm of the matrix $A$ and the operator norm of its inverse  $A^{-1}$ in ${\mathcal B}(\ell^p_w)$.
\end{abstract}

\section{Introduction}
Let ${\mathcal G}:=(V,E)$ be a connected simple graph
 with the vertex set $V$ and edge set  $E$.
 Our illustrative examples are (i) the $d$-dimensional lattice graph
 ${\mathcal Z}^d:=(\ZZ^d, E^d)$ where there exists an edge between  $k$ and $l\in \ZZ^d$, i.e., $(k, l)\in E^d$, if the Euclidean distance between $k$ and $l$ is one; (ii) the
 (in)finite circulant graph  ${\mathcal C}_G= (V_G, E_G)$
 associated with an abelian group
$$V_G=\Big\{ \prod_{i=1}^k  g_i^{n_i}, \ n_1, \ldots, n_k\in \ZZ\Big\}$$
generated by $G=\{g_1, \ldots, g_k\}$, where   $(\lambda, \lambda')\in E_G$
if and only if   either $\lambda ({\lambda'})^{-1}$ or $\lambda' \lambda^{-1}\in G$ \cite{ BE14,BE12, LO01,  samei19};
and  (iii) the communication graph of a spatially distributed  network (SDN)  whose agents  have limited sensing, data processing,
and communication  capacity for data transmission,
where agents are used as elements in  the vertex set
 and direct communication links between two agents as edges between two vertices  \cite{akyildiz02, CJS18, chong02,  shinsun19}.

For $1\le p<\infty$ and  a weight  $w=(w(\lambda))_{\lambda\in V}$  on the graph ${\mathcal G}$, let
$\ell^p_w:= \ell_w^p({\mathcal G})$ be the Banach space of all weighted $p$-summable sequences $c=(c(\lambda))_{\lambda\in V}$ equipped with
the standard norm   
\begin{equation*}
\|c\|_{p,w}=\Big(\sum_{\lambda \in V} |c(\lambda)|^p w(\lambda) \Big)^{1/p}.
\end{equation*}
For the trivial weight  $w_0=(w_0(\lambda))_{\lambda\in V}$, we will use the simplified notation
$\ell^p$ and $\|\cdot\|_p$ instead of  $\ell^p_w$ and $\|\cdot\|_{p, w}$, where $w_0(\lambda)=1$ for all $\lambda\in V$.
We say that  a matrix
\begin{equation}\label{def-matrix}
A:=\big(a(\lambda, \lambda') \big)_{\lambda, \lambda' \in V}
\end{equation}
on the graph ${\mathcal G}$
has {\em $\ell_w^p$-stability} if there exist two
positive constants $B_1$ and $B_2$ such that
\begin{equation}\label{stability-cond}
B_1 \|c\|_{p,w} \le \|A c\|_{p,w}\le B_2 \|c\|_{p,w}, \ \ c\in \ell^p_w
\end{equation}
\cite{akramjfa09,     shinsun19,  shincjfa09, sunxian14, tesserajfa10}.
We
call the maximal constant $B_1$   for the weighted stability inequality \eqref{stability-cond} to hold
as the {\em optimal lower $\ell^p_w$-stability
bound} of the matrix $A$ and denote by $\beta_{p, w}(A)$.
The (un)weighted stability for matrices  is an  essential  hypothesis in time-frequency analysis, applied harmonic analysis, and many other  mathematical and engineering fields  \cite{aldroubisiamreview01,  
christensenbook,  grochenigbook, moteesun17, 
sunsiam06}.  

In practical  sampling and  reconstruction on an SDN of large size,   signals  and noises  are usually contained in some range.
For  robust signal reconstruction and  noise  reduction, the sensing matrix
on the SDN is required to have stability on  $\ell^\infty$  \cite{CJS18}, however
 there are some difficulties to verify  $\ell^p$-stability  of a matrix  in a distributed  manner  for $p\ne 2$ \cite{moteesun19, sunpams10}.
For a matrix $A$ on a {\em finite} graph  ${\mathcal G}=(V, E)$,
 its weighted $\ell^p_w$-stability  are equivalent to each other for different exponents $1\le p\le\infty$ and weights $w$, since
 $\ell^p_w$ is isomorphic to $\ell^2$ for any exponent $1\le p\le\infty$ and weight $w$.
In particular, for the unweighted case one may verify that
the optimal lower stability bounds of a matrix $A$ for different exponents are comparable,
\begin{equation}\label{Mstability}
\frac{\beta_{p, w_0}(A)}{\beta_{q, w_0}(A)}\le  M^{|1/p-1/q|}, \ 1\le p, q\le\infty,
\end{equation}
where $M=\# V$ is the number of vertices of the graph ${\mathcal G}$. The above  estimation on  
optimal lower stability bounds for different exponents is unfavorable for matrices of large size, but it can be improved  only if
 the matrix $A$ has some additional property, such as off-diagonal decay.
 For an {\em infinite}  matrix $A=(a(i,j) )_{i,j\in{\mathbb Z}^d}$
 in the
Baskakov-Gohberg-Sj\"ostrand   algebra,  it is proved in \cite{akramjfa09, shincjfa09, tesserajfa10} that its unweighted  stabilities
are equivalent to each other for all exponents, i.e., for all $1\le p, q<\infty$,
\begin{equation*}
\beta_{q, w_0}(A)>0 \ {\rm if \ and \ only \ if} \ \beta_{p, w_0}(A)>0.
\end{equation*}
In \cite{sunca11},  Beurling algebras  of infinite  matrices $A=(a(i,j) )_{i,j\in{\mathbb Z}^d}$ are introduced. Comparing with the Baskakov-Gohberg-Sj\"ostrand  algebras, matrices in the Baskakov-Gohberg-Sj\"ostrand  algebra (resp. the Beurling algebra)  are dominated by a bi-infinite Toeplitz matrix associated with a  (resp. radially decreasing) sequence with certain decay,
and they are bounded linear operators on unweighted sequence spaces $\ell^p_{w_0}$ (resp. on weighted spaces $\ell^p_w$ for all Mukenhoupt $A_p$-weights $w$).
For an  infinite matrix in a Beurling algebra,  its weighted stabilities for different exponents and Muckenhoupt weights are established
in \cite{sunca11},
\begin{equation*}
\beta_{p, w}(A)>0 \ {\rm if \ and \ only \ if} \ \beta_{q, w'}(A)>0
\end{equation*}
where $1\le p, q<\infty$ and $w, w'$ are Muckenhoupt $A_p$- and $A_q$-weights respectively.
Obviously, the lattice $\ZZ^d$ is the vertex set of the lattice graph ${\mathcal Z}^d$.
  Inspired by the above observation,
  Beurling algebras ${\mathcal B}_{r,\alpha}({\mathcal G})$ of matrices $A=(a(\lambda, \lambda') )_{\lambda, \lambda' \in V}$
  on an arbitrary simple graph ${\mathcal G}=(V, E)$ are introduced in \cite{shinsun19}, where $1\le r\le \infty$ and $\alpha\ge 0$.
In \cite{shinsun19},
  unweighted stabilities of a matrix $A\in {\mathcal B}_{r,\alpha}({\mathcal G})$  for different exponents are shown to be equivalent to each other, where $1\le r\le \infty, \alpha>d_{\mathcal G} (1-1/r)$ and  $d_{\mathcal G}$ is the Beurling dimension of the graph ${\mathcal G}$.
  Moreover  we have the following polynomial control on its optimal lower stability bounds for different exponents,
\begin{equation}\label{polynomialunweighted.eq}
\frac{\beta_{p, w_0}(A)}{\beta_{q, w_0}(A)}
 \le  D_1
\Big(\frac{\|A\|_{{\mathcal B}_{r, \alpha}}}{\beta_{p, w_0}(A)}\Big)^{D_0|1/p-1/q|},\ \ 1\le p, q<\infty,
\end{equation}
where
$D_0, D_1$ are  absolute constants independent of matrices $A$ and the size $M$ of the graph ${\mathcal G}$.
In   the first part of this paper,
 we establish a
polynomial control property for a matrix $A\in {\mathcal B}_{r, \alpha}({\mathcal G})$
on the optimal lower weighted stability bounds for different exponents   
and  Muckenhoupt weights, see Theorem \ref{mainthm}  and Remark \ref{mainthm.rem} in Section \ref{lowerbound.section},

Let ${\mathcal B}(\ell^p_w)$ be the Banach algebra of all  matrices $A$ which are bounded operators
on
the weighted vector space $\ell^p_w$ and denote the norm of $A\in  {\mathcal B}(\ell^p_w)$ by $\|A\|_{{\mathcal B}(\ell^p_w)}$.
The weighted $\ell^p_w$-stability of a matrix  $A$ is usually considered as a weak notion of its invertibility, since
$$
\beta_{p, w}(A) \ge  \big(\|A^{-1} \|_{{\mathcal B}(\ell^p_w)}\big)^{-1}
$$
when the matrix  $A$  is  invertible  in $\ell^p_w$.  However for a matrix $A$ in a Beurling algebra,
we discover that its weighted stability in $\ell^{p}_w$ implies the existence of its  ``inverse"
$B=(b(\lambda, \lambda'))_{\lambda, \lambda'\in V}$ in the same Beurling algebra
such that
\begin{equation}\label{crucialestimate}
|c(\lambda)|\le \sum_{\lambda'\in V} |b(\lambda, \lambda')| (Ac)(\lambda')|, \ \lambda\in V,
\end{equation}
hold for all  vectors $c=(c(\lambda))_{\lambda\in V}\in \ell^{p}_w$, see Lemma \ref{tech.lem2}.
The above estimate is crucial  for us to discuss polynomial  control on optimal lower weighted stability bounds for different exponents and Muckenhoupt weights, and also to establish  norm-controlled inversion of Beurling algebras in   ${\mathcal B}(\ell^p_w)$
in the second topic of this paper.

\bigskip
Given two Banach algebras ${\mathcal A}$ and ${\mathcal B}$ with common identity such that ${\mathcal A}$ is a Banach subalgebra of ${\mathcal B}$, we say that
 ${\mathcal A}$ is {\it inverse-closed} in ${\mathcal B}$ if $A\in {\mathcal A}$ and
$A^{-1} \in {\mathcal B}$ implies $A^{-1}\in {\mathcal A}$ \cite{baskakov90, gltams06, jaffard90, sjostrand94,   sunca11, suntams07, suncasp05, wiener32}. An equivalent condition for the inverse-closedness of
${\mathcal A}$  in ${\mathcal B}$ is that given an $A\in {\mathcal A}$, its spectral sets $\sigma_{\mathcal A}(A)$ and $\sigma_{\mathcal B}(A)$
 in  Banach algebras ${\mathcal A}$ and ${\mathcal B}$ are the same, i.e.,
$$
\sigma_{\mathcal A}(A)=\sigma_{\mathcal B}(A)\ \ {\rm for \ all} \ A\in {\mathcal A}.
$$
In this paper, we also call the inverse-closed property for a Banach
subalgebra as Wiener's lemma for that subalgebra \cite{sunca11, suntams07, suncasp05, wiener32}.
For algebras of matrices  with certain off-diagonal decay, Wiener's lemma
 can be informally interpreted as localization preservation under inversion.
Such a localization preservation is of great importance in  applied harmonic analysis, numerical analysis,
and many mathematical and engineering fields, see the survey papers \cite{grochenig10, Krishtal11, shinsun13}
and  references therein for  historical remarks.
We remark that Wiener's lemma does not provide a norm estimate for the inversion, which is  essential 
for some mathematical and engineering  applications.

We say that a Banach subalgebra
 ${\mathcal A}$  of ${\mathcal B}$  admits {\em
norm-controlled inversion} in ${\mathcal B}$ if there exists a continuous function $h$ from
$[0, \infty)\times [0, \infty)$ to $[0, \infty)$
 such that
\begin{equation}\label{normcontrol}
\|A^{-1}\|_{\mathcal A}\le h\big(\|A\|_{\mathcal A}, \|A^{-1}\|_{\mathcal B}\big)
\end{equation}
for all $A\in {\mathcal A}$ being invertible in ${\mathcal B}$
\cite{gkII, gkI,  nikolski99, samei19, shinsun19}.  By the norm-controlled inversion \eqref{normcontrol}, we have
the following estimate for the resolvent of $A\in {\mathcal A}$,
\begin{equation}\label{normcontrolresolvent}
\|(\lambda  I-A)^{-1}\|_{\mathcal A}\le h\big(\|\lambda I-A\|_{\mathcal A}, \|(\lambda I-A)^{-1}\|_{\mathcal B}\big), \ \ \lambda\not\in \sigma_{\mathcal B}(A)=\sigma_{\mathcal A}(A),
\end{equation}
where $I$ is the common identity of Banach algebras ${\mathcal A}$ and ${\mathcal B}$.
The norm-controlled inversion is  a strong version of Wiener's lemma.
The classical Wiener algebra of periodic functions with summable Fourier coefficients is an inverse-closed subalgebra of the  Banach algebra
 of  all periodic continuous functions
  \cite{wiener32},  however it does not have norm-controlled inversion \cite{belinskiijfaa97, nikolski99}.
We say that
${\mathcal A}$ is  a {\em differential subalgebra of order $\theta\in (0, 1]$} in ${\mathcal B}$  
if there exists a positive constant $D:=D({\mathcal A}, {\mathcal B}, \theta)$ such that
\begin{equation}\label{differentialnorm.def}
\|AB\|_{\mathcal A}\le D\|A\|_{\mathcal A} \|B\|_{\mathcal A} \Big (\Big(\frac{\|A\|_{\mathcal B}}{\|A\|_{\mathcal A}}\Big)^\theta +
\Big(\frac{\|B\|_{\mathcal B}}{\|B\|_{\mathcal A}}\Big)^\theta
\Big)
\quad {\rm for \ all} \  A, B \in {\mathcal A}.
\end{equation}
The concept of   differential subalgebras of order $\theta$  was introduced in \cite{blackadarcuntz91, kissin94, rieffel10} for $\theta=1$ and
\cite{christ88, gkI,     shinsun19} for $\theta\in (0, 1)$.
 It has been proved that
 a differential $*$-subalgebra  ${\mathcal A}$  of a symmetric $*$-algebra ${\mathcal B}$ has norm-controlled inversion in ${\mathcal B}$
\cite{gkII, gkI,samei19, shinsun20,  suncasp05}.
A crucial step in the
proof  is to introduce
  $B:=I- \|A^* A\|_{\mathcal B}^{-1} A^* A$ for any $A\in {\mathcal A}$ being invertible in ${\mathcal B}$, whose spectrum is contained in an interval
   on the positive real axis.
   The above reduction  depends on the requirements that
   ${\mathcal B}$ is symmetric and both ${\mathcal A}$ and ${\mathcal B}$ are $*$-algebras with common identity and involution $*$.

   Several algebras of localized matrices with certain off-diagonal decay, including the
Gr\"ochenig-Schur algbera, Baskakov-Gohberg-Sj\"ostrand  algebra, Beurling algebra and Jaffard algebra,
have been shown to be differential $*$-subalgebras of the symmetric $*$-algebra ${\mathcal B}(\ell^2)$,
and hence they admit norm-controlled inversion in ${\mathcal B}(\ell^2)$
 \cite{gkII, gkI,   grochenigklotz10,   jaffard90, rssun12, samei19, shinsun19, sunca11,   suntams07,  suncasp05}.
 In \cite{gkII, gkI, shinsun19}, the authors show that for the  Baskakov-Gohberg-Sj\"ostrand  algebra, Jaffard algebra, and
 Beurling algebra of matrices, a bivariate polynomial can be selected to be the norm-control function $h$
 in \eqref{normcontrol}.

For applications in some mathematical and
engineering fields, the  widely-used algebras ${\mathcal B}$ of infinite matrices
are  the operator algebras ${\mathcal B}(\ell^p_w), 1\le p\le \infty$, which are symmetric only when $p=2$.
To our knowledge, there is no literature on  norm-controlled inversion in  a nonsymmetric algebra.
In this paper, we prove that  Beurling algebras of  localized matrices admit norm-controlled inversion in  ${\mathcal B}(\ell^p_w)$
for all exponent $1\le p<\infty$ and  Muckenhoupt $A_p$-weights $w$, and that the Beurling algebra norm of  the inversion of a matrix  $A$ is bounded by a bivariate polynomial of its Beurling algebra norm of the matrix $A$ and the operator norm of its inverse  $A^{-1}$ in ${\mathcal B}(\ell^p_w)$, see Theorem \ref{norm-contro-inversion.thm} and Remark \ref{norm-contro-inversion.thm.remark}.

The paper is organized as follows. In Section \ref{Preliminaries.section}, we recall some
preliminary results on a connected  simple graph ${\mathcal G}$,  Beurling
 algebras  of matrices on the graph ${\mathcal G}$ and on its maximal disjoint sets,  
and  weighted norm inequalities for matrices in a Beurling algebra. For matrices in a Beurling
   algebra, we consider the equivalence of their weighted stability
 for different exponents $1\le p< \infty$ and  Muckenhoupt $A_p$-weights $w$ in Section \ref{lowerbound.section},
 and their norm-controlled inversion in ${\mathcal B}(\ell^p_w)$ in Section \ref{normedinversion.section}.
 All proofs, except the proof of Theorem \ref{mainthm} in Section \ref{lowerbound.section},  are collected in Section \ref{proofs.section}.

Notation: For a real number $t$, we use the standard notation $\lfloor t\rfloor$ and $\lceil t\rceil$  to denote its floor and ceiling, respectively.  For two terms $A$ and $B$, we write $A\lesssim B$ if $A\le CB$ for some absolute constant  $C$, and
$A\approx B$ if $A\lesssim B$  and $B\lesssim A$.

\section{Preliminaries}\label{Preliminaries.section}

In Section \ref{graph.subsection},  we recall the doubling property for the counting measure $\mu$ on a connected simple graph ${\mathcal G}$
\cite{
CJS18,
shinsun19, YYH13book},
show that the counting measure $\mu$ has the strong polynomial growth property \eqref{strongpolynomial-growth}, and then define generalized Beurling dimension of the graph ${\mathcal G}$.
In Sections \ref{beurling.subsection} and \ref{beurlingmaximal.subsection},
we  recall the definition of  two closely-related Beurling algebras of matrices on
the graph ${\mathcal G}$ and on its maximal disjoint sets \cite{beurling49,  shinsun19, sunca11}, and
 provide some algebraic and approximation properties of those two Banach algebras
 of matrices. In Section  \ref{weighted.subsection}, we prove 
 that any matrix in a Beurling algebra
 is a  bounded  linear operator on weighted vector spaces $\ell^p_w$ for all $1\le p<\infty$ and
 Muckenhoupt $A_p$-weights $w$.

\subsection{Generalized Beurling dimension of a connected simple graph}\label{graph.subsection}
Let $\rho$ be
the {\em geodesic distance}  on  the connected simple graph ${\mathcal G}$, which is the nonnegative function  on $V\times V$ such that $\rho(\lambda,\lambda)=0, \lambda\in V$,  and  $\rho(\lambda,\lambda')$ is the number of edges in a shortest path
connecting distinct vertices $\lambda, \lambda'\in V$ \cite{chungbook}.
 This geodesic distance $\rho$ is a metric on $V$ of a connected simple
graph ${\mathcal G}$. For the lattice graph ${\mathcal Z^d}$, one may verify that its geodesic distance between two points $k=(k_1,..., k_d)$
 and $\ell=(\ell_1,..., \ell_d)$ is given  by $\rho(k, \ell):=\sum_{j=1}^d|k_j-\ell_j|$;
 for the circulant graph
${\mathcal C}_G$ generated by $G=\{g_1, \ldots, g_k\}$,  we have
$$\rho(\lambda, \lambda')= \inf \Big\{\sum_{i=1}^k |n_i|, \  \lambda' \lambda^{-1}=\prod_{i=1}^k g_i^{n_i}, n_1, \ldots, n_k\in \ZZ\Big\};$$
and for the communication graph of an  SDN, $\rho(\lambda, \lambda')$ is the time delay of data transmission between two agents $\lambda$ and $\lambda'$.
Using the geodesic distance  $\rho$, we define the closed ball with center $\lambda\in V$ and radius $r>0$ by
$$B(\lambda, r)= \{\lambda'\in V, \ \rho(\lambda, \lambda')\le r\},$$
which contains all  $r$-neighboring vertices of  $\lambda\in V$.

Let $\mu$ be the  counting measure on  the vertex set $V$, i.e.,
$\mu(F)$ is  the number of vertices in $F\subset V$.
In this paper, we always assume that the counting measure $\mu$
has {\em doubling property}, i.e.,  there exists a positive constant $D$ such that
\begin{equation}\label{doubling}
\mu\big(B(\lambda, 2r)\big) \le D \mu\big(B(\lambda, r)\big) \ \ \text{for all }   \lambda \in V \text{ and } r > 0
\end{equation}
\cite{CJS18, shinsun19, YYH13book}. We denote the minimal constant $D$ in the doubling property \eqref{doubling}  by $D(\mu)$, which is  also known as the {\em doubling constant} of the measure $\mu$.  Applying the doubling property \eqref{doubling}  repeatedly, we have
\begin{equation}\label{doubling2}
\mu(B(\lambda, r))\le \mu\big(B(\lambda, 2^{\lceil \log_2 (r/r')\rceil} r')\big) \le D(\mu) (r/r')^{\log_2 D(\mu)}\mu(B(\lambda, r')), \ r\ge  r'> 0.
\end{equation}
Taking $r'=1-\epsilon$ in \eqref{doubling2} for sufficiently small $\epsilon>0$, we conclude that
the counting measure $\mu$ has {\em polynomial growth} in the sense that
\begin{equation}\label{polynomial-growth}
\mu(B(\lambda, r))\le D_1 (r+1)^{d_1}\ \ \text{for all } \  \lambda \in V \text{ and } r\ge 0,
\end{equation}
where  $D_1$ and $d_1$ are positive constants.  The notion of polynomial growth for the counting measure  $\mu$  
is introduced in \cite{CJS18}, where the minimal constants $d_1$ and $D_1$ in \eqref{polynomial-growth}, to be denoted by $d_{\mathcal G}$ and $D_{\mathcal G}$,
  are known as the {\em Beurling
dimension} and {\em density} of the graph ${\mathcal G}$ respectively.

Let $N\ge 0$.  We say that  a set $V_N\subset V$ of fusion vertices is   {\it maximal $N$-disjoint} if
\begin{equation}\label{maximum1}
B(\lambda, N)\cap \big(\cup_{\lambda_m \in{V_N}} B(\lambda_m, N)\big) \ne \emptyset
\  \text{ for all } \lambda \in V
\end{equation}
and
\begin{equation}\label{maximum2}
B(\lambda_m, N)\cap B(\lambda_n, N) =\emptyset \ \text{ for all distinct }  \lambda_m, \lambda_n\in V_N.
\end{equation}
For $N=0$, one may verify that the whole set $V$ is the only  maximal $N$-disjoint set $V_N$,  i.e.,
\begin{equation}\label{V1.pro}
V_N=V \ {\rm if} \ N=0,
\end{equation}
while for $N\ge 1$, one may construct
many maximal $N$-disjoint sets $V_N$. For example, we can construct a maximal $N$-disjoint set $V_N = \{\lambda_m, m\ge 1\}$
by  taking a vertex $\lambda_1\in V$ and defining vertices $\lambda_m, m\ge 2$, recursively by
$\lambda_m= {\rm arg \ min}_{\lambda\in A_m}  \rho(\lambda, \lambda_1),$
 where $A_m=\{\lambda\in V, B(\lambda, N)\cap \cup_{m'=1}^{m-1}
B(\lambda_{m'}, N)=\emptyset\}$ \cite{CJS18}.
For a maximal $N$-disjoint set $V_N$ of fusion vertices,
it is observed in \cite{CJS18, shinsun19} that for any $N'\ge 2N$,  $B(\lambda_m, N'), \lambda_m\in V_N$, form a finite covering
of the whole set $V$,
and
\begin{equation}\label{overlap-counting}
1  \le   \inf_{\lambda \in V} \sum_{\lambda_m \in V_N}\chi_{B(\lambda_m, N')}(\lambda)
\le \sup_{\lambda \in V} \sum_{\lambda_m \in V_N}\chi_{B(\lambda_m, N')}(\lambda)
\le \big(D(\mu)\big)^{\lceil \log_2(2N'/N+1)\rceil}.
\end{equation}
For $\lambda\in V$ and $R\ge 0$, set
\begin{equation}\label{AR.def}
A_R(\lambda, N):=\big\{ \lambda_m \in V_N :\  \rho(\lambda_m, \lambda) \le (N+1)R \big\}, \end{equation}
and let $\lambda_{m_0}\in A_R(\lambda, N)$ be so chosen that
\begin{equation}\label{lambdam0.def0}
\mu(B(\lambda_{m_0}, N))= \inf_{\lambda_m\in A_R(\lambda, N)} \mu(B(\lambda_m, N)).
\end{equation}
Then we obtain from \eqref {doubling2}, \eqref {maximum2} and \eqref{lambdam0.def0} that
\begin{eqnarray}\label{counting}
 \mu(A_R(\lambda, N))
  & \le &  \frac{ \sum_{\lambda_m\in A_R(\lambda, N)}  \mu(B(\lambda_m, N))}{\mu(B(\lambda_{m_0}, N))}
 = \frac{ \mu\big( \cup_{\lambda_m\in A_R(\lambda, N)} B(\lambda_m, N)\big)}{\mu(B(\lambda_{m_0}, N))}
 \nonumber\\
& \le &
\frac{  \mu(B(\lambda_{m_0}, N+2(N+1)R)}{\mu(B(\lambda_{m_0}, N))}
\le (D(\mu))^3 (R+1)^{\log_2 D(\mu)}.
\end{eqnarray}
 Therefore  the counting measure $\mu$  on the graph ${\mathcal G}$ has {\em strong polynomial growth} since  there exist two positive constants $D$ and $d$ such that
\begin{equation}\label{strongpolynomial-growth}
\sup_{\lambda\in V} \mu\big(\big\{ \lambda_m \in V_N : \ \rho(\lambda_m, \lambda) \le (N+1)R \big\}\big) \le D (R+1)^{d}
\end{equation}
hold for all $R, N\ge 0$ and maximal $N$-disjoint set $V_N$ of fusion vertices.
Recall that  the whole set $V$ is the only  maximal $N$-disjoint set $V_N$ for $N=0$.
So in this paper the minimal constants $d$ and $D$ in \eqref{strongpolynomial-growth}, to be denoted by $\tilde d_{\mathcal G}$ and $\tilde  D_{\mathcal G}$,
are considered as  {\em generalized Beurling dimension} and {\em density} respectively. Moreover
it follows from \eqref{V1.pro} and \eqref{counting} that
\begin{equation}
d_{\mathcal G}\le \tilde d_{\mathcal G}
\le \log_2 D(\mu)
\end{equation}
where $d_{\mathcal G}$ is the Beurling dimension of the graph ${\mathcal G}$.

We say that
the counting measure $\mu$ on the graph ${\mathcal G}$ is {\em Ahlfors $d_0$-regular} if
 there exist   positive constants $B_3$ and $B_4$ such that
 \begin{equation}\label{ahlfors.def}
  B_3 (r+1)^{d_0}\le \mu \big(B(\lambda, r)\big)\le B_4 (r+1)^{d_0}\end{equation}
hold for  all balls $B(\lambda, r)$ with  center $\lambda\in V$ and radius $0\le r \le {\rm diam}\  {\mathcal G}$, where ${\rm diam}\ {\mathcal G}$ denotes the diameter of the graph ${\mathcal G}$
\cite{Keith2008,  YYH13book}.
Clearly for a graph ${\mathcal G}$ with its counting measure $\mu$ being Ahlfors $d_0$-regular,
its  Beurling dimension $d_{\mathcal G}$ is equal to $d_0$.
In the following proposition, we show that the  generalized Beurling dimension  $\tilde d_{\mathcal G}$  is also equal to $d_0$, see Section \ref{regular.prop.pfsection} for the proof.

\begin{prop}\label{regular.prop}
Let ${\mathcal G}$ be a connected simple graph. If
the counting measure $\mu$ is Ahlfors $d_0$-regular, then   $\tilde d_{\mathcal G}=d_0$.
\end{prop}

\subsection{Beurling algebras of matrices on graphs}
\label{beurling.subsection}

Let  ${\mathcal G}:=(V, E)$ be a connected simple graph with its counting measure $\mu$
satisfying the doubling property \eqref{doubling}.
For $1 \le r \le \infty$ and $\alpha \ge 0$, we define the Beurling algebra
${\mathcal B}_{r, \alpha}:={\mathcal B}_{r,\alpha}({\mathcal G})$
by
\begin{equation}\label{beurling.def}
{\mathcal B}_{r, \alpha}({\mathcal G}):=\Big\{A=\big(a(\lambda, \lambda') \big)_{\lambda, \lambda' \in V}:  \  \  \|A\|_{{\mathcal B}_{r, \alpha}}<\infty\Big\},
\end{equation}
where  $d_{\mathcal G}$ is the Beurling dimension of the graph ${\mathcal G}$,  $h_A(n)=\sup_{\rho(\lambda, \lambda')\ge n} |a(\lambda, \lambda')|, n\ge 0$,
and 
\begin{equation}\label{bralpha.norm}
\|A\|_{{\mathcal B}_{r, \alpha}}:=\left\{\begin{array}{ll}
\big(\sum_{n=0}^\infty h_A(n)^r (n+1)^{\alpha r+d_{\mathcal G}-1} \big)^{1/r}  &  {\rm if} \ 1\le  r<\infty\\[5pt]
 \sup_{n\ge 0} h_A(n) (n+1)^\alpha
  & {\rm if} \ r=\infty.
  \end{array}
  \right.
\end{equation}
The Beurling algebra ${\mathcal B}_{r,\alpha}({\mathcal G})$
is introduced in \cite{sunca11} for the lattice graph ${\mathcal Z}^d$ and for
an arbitrary simple graph ${\mathcal G}$ in \cite{shinsun19}.
For a matrix $A=(a(\lambda, \lambda'))_{\lambda, \lambda'\in V}$  in the Beurling algebra $ {\mathcal B}_{r,\alpha}({\mathcal G})$,
we define  approximation matrices $A_K, \ K\ge 1$, with finite bandwidth  by
\begin{equation}\label{defAN}
A_K:=\big( a(\lambda, \lambda')
\chi_{[0,1]}(\rho(\lambda, \lambda')/K)\big)_{\lambda, \lambda'\in V}.
\end{equation}
For the Beurling algebra ${\mathcal B}_{r, \alpha}({\mathcal G})$,  we recall some elementary properties
where  the first four conclusions have been established in \cite{shinsun19}, see Section \ref{beurling.prop.pfsection} for the proof.

\begin{prop}  
\label{beurling.prop}
Let  ${\mathcal G}:=(V, E)$ be a connected simple graph such that its counting measure $\mu$
satisfies the doubling property \eqref{doubling} with the doubling constant  $D(\mu)$.
Then the following  statements hold.

\begin{itemize}
\item[{(i)}] ${\mathcal B}_{r, \alpha}({\mathcal G})$ with $1\le r\le \infty$ and $\alpha\ge 0$ are solid in the sense that
\begin{equation}\label{solid.eq}
\|A\|_{{\mathcal B}_{r, \alpha}}\le \|B\|_{{\mathcal B}_{r, \alpha}}\end{equation}
hold for all $A=(a(\lambda, \lambda'))_{\lambda, \lambda'\in V}$ and $B=(b(\lambda, \lambda'))_{\lambda, \lambda'\in V}$
satisfying
$|a(\lambda, \lambda')|\le |b(\lambda, \lambda')|$ for all $\lambda, \lambda'\in V.$

\item[{(ii)}] ${\mathcal B}_{1, 0}({\mathcal G})$  is  a  Banach algebra, and
\begin{equation}
\|AB\|_{{\mathcal B}_{1, 0}}\le   d_{\mathcal G}
 D_{\mathcal G}  2^{d_{\mathcal G}+1} \|A\|_{{\mathcal B}_{1, 0}} \|B\|_{{\mathcal B}_{1, 0}} \ {\rm for \ all}\  A, B\in {\mathcal B}_{1, 0}({\mathcal G}).
\end{equation}

\item [{(iii)}] ${\mathcal B}_{r, \alpha}({\mathcal G})$ with $1\le r\le \infty$ and $\alpha>d_{\mathcal G}(1-1/r)$
 are  Banach algebras,
 and
 \begin{equation} \label{beurling.prop.eq3}
        \|AB\|_{{\mathcal B}_{r, \alpha}}  
 \le  d_{\mathcal G} D _{\mathcal G}  2^{\alpha +1+d_{\mathcal G}/r}  \Big(\frac{\alpha-(d_{\mathcal G}-1)(1-1/r)}{\alpha-d_{\mathcal G}(1-1/r)}\Big)^{1-1/r}
      \|A\|_{{\mathcal B}_{r, \alpha}}\|B\|_{{\mathcal B}_{r, \alpha}}\ \   {\rm for \ all} \ A, B\in {\mathcal B}_{r, \alpha}({\mathcal G}).
\end{equation}

 \item[{(iv)}]  ${\mathcal B}_{r, \alpha}({\mathcal G})$ with $1\le r\le \infty$ and $\alpha>d_{\mathcal G}(1-1/r)$ are  Banach subalgebras of
${\mathcal B}_{1, 0}({\mathcal G})$, and  
\begin{equation} \label{beurling.prop.eq1}
\|A\|_{{\mathcal B}_{1,0}}
\le \Big( \frac{\alpha -(d_{\mathcal G}-1)(1-1/r)}{\alpha -d_{\mathcal G}(1-1/r)}\Big)^{1-1/r}
\|A\|_{{\mathcal B}_{r,\alpha}} \ {\rm for \ all} \ A\in {\mathcal B}_{r, \alpha}({\mathcal G}).
 \end{equation}

\item[{(v)}] A matrix $A$ in
 ${\mathcal B}_{r, \alpha}({\mathcal G})$ with $1\le r\le \infty$ and $\alpha>d_{\mathcal G}(1-1/r)$
 is well approximated by its truncation $A_K, K\ge 1$, in the norm $\|\cdot\|_{{\mathcal B}_{1, 0}}$,
\begin{equation}\label{AN-appr}
\|A-A_K\|_{{\mathcal B}_{1, 0}} \le  C_0 \|A\|_{{\mathcal B}_{r,\alpha}} K^{-\alpha  +d_{\mathcal G}(1-1/r)},
\end{equation}
where
$$C_0=
 \left\{\begin{array}{ll}
2^{\alpha+1 }
& {\rm if} \ r=1\\
\frac{2^{\alpha+1 -d_{\mathcal G}(1-1/r)}}{(\alpha /(1-1/r)  -d_{\mathcal G})^{1-1/r} } 
 & {\rm if} \ r>1.
\end{array} \right.
$$
%
%

\end{itemize}
\end{prop}

\subsection{Beurling algebras of matrices on a maximal disjoint set of fusion vertices}
\label{beurlingmaximal.subsection}

Given  $1\le r\le \infty, \tilde \alpha\ge 0$  and a maximal $N$-disjoint set $V_N$ of fusion vertices, we define
  Beurling algebras  of matrices $B:=\big(b(\lambda_m, \lambda_k)\big)_{\lambda_m, \lambda_k\in V_N} $ on
$V_N$ by
\begin{equation}
\mathcal{B}_{r,\tilde \alpha ;N}(V_N):=\big\{B, \
\|B\|_{\mathcal{B}_{r,\tilde \alpha;N}}<\infty\big\}
\end{equation}
where
\begin{equation}\label{N-norm}
\|B\|_{\mathcal{B}_{r,\tilde \alpha;N}}:= \left\{\begin{array}{ll}
\Big(\sum_{n=0}^\infty (n+1)^{\tilde \alpha r+\tilde d_{\mathcal G}-1} \Big(\sup_{\rho(\lambda_m, \lambda_k)\ge n (N+1)}
|b(\lambda_m, \lambda_k )|\Big)^r\Big)^{1/r}  & {\rm if} \ 1\le r<\infty\\
\sup_{n\ge 0} (n+1)^{\tilde \alpha } \big(\sup_{\rho(\lambda_m, \lambda_k)\ge n  (N+1)}
|b(\lambda_m, \lambda_k )| \big)  & {\rm if} \ r=\infty.\end{array}
\right.
\end{equation}
 The Banach algebra $\mathcal{B}_{r,\tilde \alpha;N}(V_N)$  is
 introduced in \cite{shinsun19}, where the counting measure $\mu$ is  assumed to be Ahlfors regular
 in which the generalized Beurling dimension  $\tilde d_{\mathcal G}$ and the Beurling dimension $d_{\mathcal G}$ coincides by
 Proposition \ref{regular.prop}.  Following the
  argument used in the proof of  Proposition  \ref{beurling.prop}
 with the polynomial growth property \eqref{polynomial-growth} replaced by the strong polynomial growth property
\eqref{strongpolynomial-growth}, we have the following properties for Banach algebras ${\mathcal B}_{r, \tilde \alpha; N}(V_N)$
of matrices on $V_N$.

\begin{prop}\label{VNbanach.prop}
Let  ${\mathcal G}:=(V, E)$ be a connected simple graph such that its counting measure $\mu$
satisfies the doubling property \eqref{doubling},
 and
$V_N$ be a maximal $N$-disjoint set of fusion vertices.
Then the following statements hold.

\begin{itemize}

\item[{(i)}] $\mathcal{B}_{1,0;N}(V_N)$ is a Banach algebra and
\begin{equation}\label{banach-algebra}
\|AB\|_{\mathcal{B}_{1,0;N}}\le  \tilde d_{\mathcal G} \tilde D_{\mathcal G}   2^{3\tilde d_{\mathcal G}+1} \|A\|_{\mathcal{B}_{1,0;N}}\|B\|_{\mathcal{B}_{1,0;N}}, \ A, B\in {\mathcal B}_{1, 0; N}(V_N).
\end{equation}

\item[{(ii)}]  ${\mathcal B}_{r, \tilde \alpha; N}(V_N)$ with $1\le r\le \infty$ and $\tilde \alpha > \tilde d_{\mathcal G} (1-1/r)$ are Banach subalgebras of
${\mathcal B}_{1, 0; N}(V_N)$, and
\begin{equation}\label{banach-algebra+2}
\|A\|_{\mathcal{B}_{1,0 ;N}}\le
\Big( \frac{\tilde \alpha -(\tilde d_{\mathcal G}-1)(1-1/r)}{\tilde \alpha -
\tilde d_{\mathcal G}(1-1/r)}\Big)^{1-1/r}
\|A\|_{\mathcal{B}_{r,\tilde \alpha;N}},  A\in {\mathcal B}_{r, \tilde \alpha; N}(V_N).
\end{equation}

\item [{(iii)}] ${\mathcal B}_{r, \tilde \alpha ;N}(V_N)$ with $1\le r\le \infty$ and $\tilde  \alpha> \tilde d_{\mathcal G} (1-1/r)$
are Banach algebras,  and
\begin{eqnarray}\label{banach-algebra+1}
\|AB\|_{\mathcal{B}_{r,\tilde \alpha ;N}}&\le &
\tilde d_{\mathcal G}
\tilde D_{\mathcal G}
2^{\tilde \alpha+ \tilde d_{\mathcal G}(2+1/r)+2}\Big( \frac{\tilde \alpha -(\tilde d_{\mathcal G}-1)(1-1/r)}{\tilde\alpha -
\tilde d_{\mathcal G}(1-1/r)}\Big)^{1-1/r}
\nonumber \\&  &
\times \|A\|_{\mathcal{B}_{r,\tilde \alpha;N}}\|B\|_{\mathcal{B}_{r, \tilde \alpha;N}},
\ \ A, B\in
 {\mathcal B}_{r, \tilde \alpha; N}(V_N).
\end{eqnarray}
\end{itemize}
\end{prop}

Beurling algebra on the graph ${\mathcal G}$ and on its maximal  $N$-disjoint set $V_N$ of fusion vertices are closely related.
For $N=0$, we have
\begin{equation}
 {\mathcal B}_{r, \tilde \alpha; 0}(V_0)={\mathcal B}_{r, \tilde \alpha+ (\tilde d_{\mathcal G}-d_{\mathcal G})/r}({\mathcal G})\end{equation}
  as the only maximal $0$-disjoint set  $V_0$ is the whole vertex set $V$.
  For $N\ge 1$,  we have the following results about  Beurling algebras on a graph and its maximal disjoint sets, which will be used
  in our proofs to establish the  equivalence of weighted stability for different exponents and weights  and also
  the norm-controlled inversion.   The detailed proof 
   will be given in Section \ref{beurlingonmaximalsets.pr.pfsection}.

  \begin{prop}\label{beurlingonmaximalsets.pr}
  Let $1\le r\le \infty$,   ${\mathcal G}:=(V, E)$ be a connected simple graph such that its counting measure $\mu$
satisfies the doubling property \eqref{doubling},
 and
$V_N, N\ge 1$, be a maximal $N$-disjoint set of fusion vertices.
Then the following statements hold.
\begin{itemize}
\item[{(i)}] If $A=(a(\lambda, \lambda'))_{\lambda, \lambda'\in V}\in {\mathcal B}_{r, \alpha} ({\mathcal G}), \alpha \ge 0$, then
its submatrix
  $B= (a(\lambda_m, \lambda_k))_{\lambda_m, \lambda_k\in V_N}$ belongs to ${\mathcal B}_{r, \alpha - (\tilde d_{\mathcal G}-d_{\mathcal G})/r; N}$, and
  \begin{equation}\label{beurlingonmaximalsets.pr.eq1}
  \|B\|_{{\mathcal B}_{r, \alpha - (\tilde d_{\mathcal G}-d_{\mathcal G})/r; N}}\le \|A\|_{{\mathcal B}_{r, \alpha}}.
  \end{equation}

  \item[{(ii)}] If $B= (b(\lambda_m, \lambda_k))_{\lambda_m, \lambda_k\in V_N}\in {\mathcal B}_{r, \alpha; N}, \alpha\ge 0$, the matrix
    \begin{equation}\label{beurlingonmaximalsets.pr.eq2}
    A=\Big(\sum_{\lambda_m\in B(\lambda, 2N)} \sum_{\lambda_k\in B(\lambda', 4N)} b(\lambda_m, \lambda_k)\Big)_{\lambda, \lambda'\in V}
    \end{equation}
    on the graph ${\mathcal G}$ belongs to  $B_{r, \alpha+ (\tilde d_{\mathcal G}-d_{\mathcal G})/r}({\mathcal G})$,
    and
    \begin{equation} \label{beurlingonmaximalsets.pr.eq3}
    \|A\|_{{\mathcal B}_{r, \alpha+ (\tilde d_{\mathcal G}-d_{\mathcal G})/r}}\le 8^{\alpha +\tilde d_{\mathcal G}/r} (D(\mu))^{7} N^{\alpha +\tilde d_{\mathcal G}/r}
     \|B\|_{{\mathcal B}_{r, \alpha; N}}.
    \end{equation}

\item[{(iii)}] If $A=(a(\lambda, \lambda'))_{\lambda, \lambda'\in V}\in {\mathcal B}_{r, \alpha} ({\mathcal G})$ for some $\alpha>  d_{\mathcal G}(1-1/r)$, then
the matrix
\begin{equation}
S_{A,N}=\big(S_{A, N}(\lambda_m, \lambda_k)\big)_{\lambda_m, \lambda_k\in V_N},\end{equation}
belongs to ${\mathcal B}_{r, \alpha-(\tilde d_{\mathcal G}-d_{\mathcal G})/r; N}$, and
\begin{equation}\label{SAN.estimate}
\|S_{A, N}\|_{\mathcal{B}_{r, \alpha-(\tilde d_{\mathcal G}-d_{\mathcal G})/r;N}}\le
\tilde C_0
 \|{A}\|_{{\mathcal B}_{r,\alpha}}\times \left\{\begin{array}{ll} N^{-\min(1,\alpha-d_{\mathcal G}/r')} & {\rm if} \ \alpha\ne d_{\mathcal G}/r'+1\\
N^{-1} (\ln (N+1))^{1/r'}  & {\rm if} \ \alpha= d_{\mathcal G}/r'+1,
\end{array}\right.
\end{equation}
where  $1/r'=1-1/r$,  $h_A(n)=\sup_{\rho(\lambda, \lambda')\ge n} |a(\lambda, \lambda')|, n\ge 0$,
\begin{equation*}
\tilde C_0=  
2^{4\alpha+4d_{\mathcal G}/r+2} 
\times
\left\{
\begin{array}{ll}
\big(\frac{1+|\alpha -1-d_{\mathcal G}/r'|}{|\alpha-1-d_{\mathcal G}/r'|}\big)^{1/r'}
& {\rm if} \ \alpha\ne d_{\mathcal G}/r'+1
\\
1 & {\rm if} \ \alpha= d_{\mathcal G}/r'+1,
\end{array}\right.
\end{equation*}
 and  for $\lambda_m, \lambda_k\in V_N$,
\begin{equation}\label{VAN-1}
S_{ A, N}(\lambda_m, \lambda_k)= \left\{
\begin{array}{ll}
N^{d_{\mathcal G}}  h_{ A}\big(\rho(\lambda_m, \lambda_k)/2\big)
&{\rm if }\ \rho(\lambda_m, \lambda_k)> 12(N+1)
\\
N^{-1} \sum_{n=0}^{2N} h_A(n) (n+1)^{d_{\mathcal G}}
&
{\rm if }\ \rho(\lambda_m, \lambda_k)\le  12(N+1).
\end{array}
\right.
\end{equation}
 \end{itemize}
  \end{prop}

\subsection{Weighted norm inequalities}
\label{weighted.subsection}
Let  ${\mathcal G}:=(V, E)$ be a connected simple graph with its counting measure $\mu$
satisfying the doubling property \eqref{doubling}.
For $1\le p<\infty$, a positive function $w=(w(\lambda))_{\lambda\in V}$ on the vertex set $V$ is a  {\em Muckenhoupt $A_p$-weight}  if
there exists a positive constant $C$ such that 
\begin{equation}\label{Aq-weight1}
\Big( \frac{1}{\mu(B)} \sum_{\lambda \in B} w(\lambda)\Big)
\Big( \frac{1}{\mu(B)} \sum_{\lambda \in B} \big(w(\lambda)\big)^{-1/(p-1)}\Big)^{p-1}
\le C
\end{equation}
for $1<p<\infty$, and
a Muckenhoupt $A_1$-weight if 
\begin{equation}\label{Aq-weight2}
\frac{1}{\mu(B)} \sum_{\lambda \in B} w(\lambda) \le C \inf_{\lambda \in B} w(\lambda)
\end{equation}
for any ball $B\subset V$ \cite{garciabook}.
The smallest constant $C$ for which
\eqref{Aq-weight1} holds for $1<p<\infty$, and
\eqref{Aq-weight2} holds for $p=1$, respectively
is known as the  $A_p$-bound of the weight $w$ and is denoted by $A_p(w)$.
An equivalent definition of  a Muckenhoupt $A_p$-weight
 $w:=(w(\lambda))_{\lambda\in V}$ is that
\begin{equation}\label{Aq-char}
\Big( \frac{1}{\mu(B)} \sum_{\lambda \in B} |c(\lambda)|\Big)^p
\Big( \frac{1}{\mu(B)} \sum_{\lambda \in B} w(\lambda)\Big)
\le \frac{A_p(w)} {\mu(B)} \sum_{\lambda \in B} |c(\lambda)|^p w(\lambda)
\end{equation}
holds for all balls $B\subset V$ and sequences $c:=\big(c(\lambda)\big)_{\lambda\in V}\in \ell^p_w$, where
$A_p(w)$ is
$A_p$-bound  of the weight $w$.
For $\lambda\in V$ and $r\ge 0$, set
$$w(B(\lambda, r))=\sum_{\lambda'\in B(\lambda, r)} w(\lambda').$$
It is well known that a Muckenhoupt $A_p$-weight  $w$ is a doubling measure. In fact, replacing the ball   $B$  and the sequence $c$ by $B(\lambda, 2^jr), 1\le j\in \ZZ$ and
the index sequence on $B(\lambda, r)$ in \eqref{Aq-char} and using the doubling condition
\eqref{doubling} for the counting measure $\mu$, we obtain that
\begin{equation}\label{weighteddoubling}
w(B(\lambda, 2^jr))\le  A_p(w) \Big(\frac{\mu(B(\lambda, 2^jr)}{\mu(B(\lambda, r)}\Big)^p w(B(\lambda, r))\le
(D(\mu))^{jp} A_p(w) w(B(\lambda, r))
\end{equation}
hold for all $\lambda\in V, r\ge 0$ and positive integers $j$.

Weighted norm inequalities of linear operators are an important topic in harmonic analysis, see \cite{garciabook} and references therein
for historical remarks.
In the following proposition, we show that the Banach algebra ${\mathcal B}_{1, 0}({\mathcal G})$
 is
a Banach subalgebra of ${\mathcal B}(\ell^p_w)$, see Section \ref{weighted.prop.pfsection} for the proof.

\begin{prop}\label{weighted.prop}
Let  ${\mathcal G}:=(V, E)$ be a connected simple graph such that its counting measure $\mu$
satisfies the doubling property \eqref{doubling}.
Then
${\mathcal B}_{1,0}({\mathcal G})$  is a subalgebra of ${\mathcal B}(\ell_w^p)$ for any $1\le p<\infty$ and Muckenhoupt $A_p$-weight $w$,
and
\begin{equation}\label{continuity}
\|Ac\|_{p,w} \le  2^{3d_{\mathcal G}}   D_{\mathcal G}  (A_p(w))^{1/p}
\|A\|_{{\mathcal B}_{1,0}} \|c\|_{p, w}
\ \ {\rm for \ all} \  A\in {\mathcal B}_{1, 0}({\mathcal G})\ {\rm and}\  c \in \ell_w^p.
\end{equation}
\end{prop}

By Propositions \ref{beurling.prop} and \ref{weighted.prop}, we conclude that
   ${\mathcal B}_{r, \alpha}({\mathcal G})$ with $1\le r\le \infty$ and $\alpha>d_{\mathcal G}(1-1/r)$ are Banach subalgebras of ${\mathcal B}(\ell^p_w)$ too.
We remark that the  subalgebra property in  Proposition \ref{weighted.prop}
was established  in \cite{CJS18, shinsun19}   for the unweighted case  and
in \cite{sunca11} for the weighted case on the lattice  graph ${\mathcal Z}^d$.

\section{Polynomial control on  optimal lower stability bounds}\label{lowerbound.section}

In this section, we show that weighted stabilities of  matrices in  a  Beurling algebra for different exponents and Muckenhoupt weights are equivalent to each other, and  reciprocal of the optimal lower stability bound for one exponent and  weight is  dominated by a polynomial of
reciprocal of
 the optimal lower stability bound for another exponent and weight.

\begin{thm}\label{mainthm}
Let $ 1 \le r \le \infty,  \ 1\le  p, q<\infty$,
${\mathcal G}:=(V,E)$ be a connected simple graph satisfying the doubling property \eqref{doubling},
$w, \, w'$  be
Muckenhoupt $A_p$-weight and $A_q$-weight respectively,
 and 
 let $A \in {{\mathcal B}_{r,\alpha}}({\mathcal G})$ for some $\alpha >\tilde d_{\mathcal G}- d_{\mathcal G}/r$,
 where $d_{\mathcal G}$ and $\tilde d_{\mathcal G}$ are the Beurling and generalized Beurling dimension
 of the graph  ${\mathcal G}$ respectively.
If  $A$ has $\ell^p_w$-stability with the optimal lower stability bound $\beta_{p, w}(A)$,
\begin{equation}\label{mainthm.eq-1}
\beta_{p, w}(A) \|c\|_{p,w} \le \|A c\|_{p,w} \ \ {\rm for \ all} \  c\in \ell^p_w,
\end{equation}
 then  $A$ has $\ell^{q}_{w'}$-stability
 with the optimal lower stability bound  denoted by $\beta_{q, w'}(A)$,
 \begin{equation}\label{mainthm.eq0}
\beta_{q, w'}(A) \|c\|_{q,w'} \le \|A c\|_{q,w'} \ \ {\rm for \ all}\ c\in \ell^{q}_{w'}.
\end{equation}
 Moreover, there exists an absolute constant $C$, independent of matrices
$A\in {{\mathcal B}_{r,\alpha}}({\mathcal G})$ and weights $w, \, w'$, such that
\begin{eqnarray}\label{lower-bound}
\frac{ \beta_{p, w}(A)}{ \beta_{q,w'}(A)} & \le  &  C
\big(A_{q}(w')\big)^{1/q} \big(A_p(w)\big)^{1/p}
  \Big( \frac{\big(A_p(w)\big)^{2/p} \|A\|_{{\mathcal B}_{r,\alpha}}}{\beta_{p, w}(A)}\Big)^{E(\alpha, r, d_{\mathcal G})
}\nonumber\\
 & & \times
 \left\{
 \begin{array}{ll}  1
 & {\rm  if} \ \alpha\ne 1+d_{\mathcal G}/r'\\
 \Big(\ln\Big(\frac{(A_p(w))^{2/p}
 \|A\|_{{\mathcal B}_{r,\alpha}}}{\beta_{p,w}(A)}\Big)\Big)^{(2 d_{\mathcal G}+1)/r'}
 & {\rm if} \ \alpha=1+d_{\mathcal G}/r'
 \end{array}\right.
\end{eqnarray}
where  $1/r'=1-1/r$ and
$$E(\alpha, r,  d_{\mathcal G})=\frac{\tilde d_{\mathcal G}+d_{\mathcal G}+1}{\min\big(1,\alpha-d_{\mathcal G}/r'\big)}.$$
\end{thm}

\begin{rem}\label{mainthm.rem}
{\rm
 The equivalence of unweighted stabilities for different exponents
 is discussed
 for matrices  in
Baskakov-Gohberg-Sj\"ostrand  algebras, Jaffard algebras  and  Beurling algebras \cite{akramjfa09,  shincjfa09,sunca11, tesserajfa10},
for  convolution operators \cite{barnes90},
and  for localized  integral operators of non-convolution type  
\cite{fang18, fang17, rssun12, shincjfa09}.
For a matrix $A$ in the Beurling algebra ${\mathcal B}_{r, \alpha}$ with $1\le r\le \infty$ and $\alpha>d_{\mathcal G}(1-1/r)$,
Shin and Sun use the boot-strap argument in \cite{shinsun19} to prove  that
 reciprocal of its optimal lower unweighted stability bound for one exponent is dominated by a polynomial of reciprocal of its
optimal lower unweighted stability bound for another exponent,
\begin{equation}\label{stability.oldthm.eq2}
\frac{\|A\|_{{\mathcal B}_{r, \alpha}}}{\beta_{q, w_0}}
 \le  C  \left\{\begin{array}
{ll} \Big(\frac{\|A\|_{{\mathcal B}_{r, \alpha}}}{\beta_{p, w_0}}\Big)^{(1+\theta(p, q))^{K_0}}
 & {\rm if} \ \alpha\ne 1+d_{\mathcal G}/r'\\
\Big(\frac{\|A\|_{{\mathcal B}_{r, \alpha}}}{\beta_{p, w_0}}
\ln \big(1+\frac{\|A\|_{{\mathcal B}_{r, \alpha}}}{\beta_{p, w_0}}\big)\Big)^{(1+\theta(p, q))^{K_0}}
  &  {\rm if} \ \alpha=1+d_{\mathcal G}/r',
  \end{array}\right.
\end{equation}
where  $C$ is an absolute constant, $K_0$ is a positive integer satisfying
$K_0> \frac{d_{\mathcal G}}{\min(\alpha-d_{\mathcal G}/r', 1)}$,
and  $$\theta(p, q)=\frac{d_{\mathcal G}|1/p-1/q|}{K_0 \min(\alpha-d_{\mathcal G}/r', 1)-d_{\mathcal G}|1/p-1/q|}.$$
Given $1\le p<\infty$, we remark that for an exponent $q$ close to $p$,
the conclusion \eqref{stability.oldthm.eq2} provides  a better estimate to the optimal lower unweighted stability bound
$\beta_{q, w_0}(A)$ than the one in
  \eqref{lower-bound} with  $w=w'=w_0$, while the conclusion  \eqref{lower-bound} with  $w=w'=w_0$  gives a tighter estimate
  to the optimal lower unweighted stability bound
$\beta_{q, w_0}(A)$ than  the one  in \eqref{stability.oldthm.eq2} when $q$ is  close to one or infinity.
}\end{rem}

\smallskip

For $ 1\le N\in\ZZ$ and $ \lambda \in V$, we introduce a truncation operator $\chi_\lambda^N$
and its smooth version   $\Psi_\lambda^N$ by
\begin{equation}\label{def-trunc}
\chi_{\lambda}^N:\big(c(\lambda)\big)_{{\lambda\in V}} \longmapsto
\big( \chi_{[0,N]} \big(\rho(\lambda,\lambda') \big) c(\lambda')
\big)_{{\lambda'\in V}}
\end{equation}
and
\begin{equation}\label{def-smooth-trunc}
\Psi_{\lambda}^N:\big(c(\lambda)\big)_{{\lambda\in V}}\longmapsto
\big( \psi_0 \big(\rho(\lambda,\lambda')/N  \big) c(\lambda')
\big)_{{\lambda'\in V}},
\end{equation}
where $\psi_0(t)=\max\{0,\min(1,3-2|t|)\}$ is the trapezoid function.
The  operators $\chi_{\lambda}^N$ and $\Psi_{\lambda}^N$ localize a sequence
to a neighborhood of $\lambda$ and they can be considered as  diagonal matrices with entries
$\chi_{B(\lambda,N)}(\lambda') $ and $\psi_0(\rho(\lambda,\lambda')/N), \lambda'\in V$ respectively.

Let $ V_{N}$  be a  maximal $N$-disjoint set
  of fusion vertices.
  To prove Theorem \ref{mainthm},  we  start from an estimate to the weighted terms
$\big(w\big(B(\lambda_m, 4N)\big))^{-1/p}
\|\Psi_{\lambda_m}^{2N} c\|_{p,w},\lambda_m\in V_N$, for sufficiently large $N$, which is established in
\cite{shinsun19}
 for  the trivial weight $w_0$.

\begin{lem}\label{tech.lem}
Let $1\le p<\infty,  1\le r\le \infty,  \alpha>d_{\mathcal G}(1-1/r)$,
$w$ be a Muckenhoupt $A_p$-weight,
and $A=(a(\lambda, \lambda'))_{\lambda, \lambda'\in V}\in {\mathcal B}_{r, \alpha}({\mathcal G})$ have $\ell^p_w$-stability.
Assume that  $N\ge 1$ is a positive integer such that
\begin{equation}\label{tech.lem.eq1}
2C_1  (A_p(w))^{1/p}
 \|A\|_{{\mathcal B}_{r,\alpha}} N^{-\alpha  +d_{\mathcal G}(1-1/r)}\le \beta_{p, w}(A),
\end{equation}
where $\beta_{p, w}(A)$ is the optimal lower $\ell^p_w$-stability bound, $C_0$ is the constant in \eqref{AN-appr} and $C_1= 2^{3d_{\mathcal G}}  C_0    D_{\mathcal G}$.
Then 
 for all
maximal $N$-disjoint sets
 $ V_{N}$  of fusion vertices  and weighted sequences $c\in \ell^p_w$, we have
\begin{eqnarray}\label{tech.lem.eq2}
& &   \beta_{p, w}(A)  \big(w\big(B(\lambda_m, 4N)\big))^{-1/p}
\|\Psi_{\lambda_m}^{2N} c\|_{p,w} 
 \le 2 \big(w\big(B(\lambda_m, 4N)\big))^{-1/p}
\|\Psi_{\lambda_m}^{2N} {A}c\|_{p,w} \nonumber\\
& &  \qquad\qquad
+ C_2 \big(A_p(w)\big)^{2/p}
\sum_{\lambda_k\in V_N} S_{A, N}(\lambda_m, \lambda_k)  \big(w\big(B(\lambda_k, 4N)\big))^{-1/p} \|\Psi_{\lambda_k}^{2N} c\|_{p,w},
\ \lambda_m\in V_N,
\end{eqnarray}
where the smooth truncation operators  $\Psi_{\lambda_m}^{2N}, \lambda_m\in V_N$, are defined in  \eqref{def-smooth-trunc}, 
 the matrix $S_{A, N}=(S_{A, N}(\lambda_m, \lambda_k))_{\lambda_m, \lambda_k\in V_N} $ is given in \eqref{VAN-1}, and
$C_2\ge 2$ is an absolute constant.
\end{lem}

Let $[\Psi_{\lambda_m}^{2N}, {A}]:=\Psi_{\lambda_m}^{2N}{A}- {A} \Psi_{\lambda_m}^{2N}$ be the commutator  between
the smooth truncation operator $\Psi_{\lambda_m}^{2N}$ and the matrix ${A}$  \cite{shinsun19,
shincjfa09, sjostrand94}.
A crucial step in the  proof of Lemma \ref{tech.lem} is the following estimate to the commutator $[\Psi_{\lambda_m}^{2N}, {A}]$,
$$\|\chi_{\lambda_m}^{4N } [\Psi_{\lambda_m}^{2N}, {A}]\chi_{\lambda_k}^{3N}\|_{{\mathcal B}_{1, 0}}\lesssim S_{A, N}(\lambda_m, \lambda_k),
\ \lambda_m, \lambda_k\in V_N,$$ see Section \ref{tech.lem.pfsection} for the detailed argument.

By  Propositions \ref{VNbanach.prop} and \ref{beurlingonmaximalsets.pr}, there exists an absolute constant
$C_3$ such that
\begin{equation}\label{norm-power}
\big\|\big(S_{A, N}\big)^{\ell}\big\|_{\mathcal{B}_{r, \alpha-(\tilde d_{\mathcal G}-d_{\mathcal G})/r;N}} \le
(C_3\|{A}\|_{{\mathcal B}_{r,\alpha}})^l \times \left\{\begin{array}{ll} N^{-\min(1,\alpha-d_{\mathcal G}/r')l } & {\rm if} \ \alpha\ne d_{\mathcal G}/r'+1\\
N^{-l} (\ln (N+1))^{l/r'}  & {\rm if} \ \alpha= d_{\mathcal G}/r'+1
\end{array}\right.
\end{equation}
hold for all $l\ge 1$, where $1/r'=1-1/r$.
 Applying \eqref{tech.lem.eq2} and \eqref{norm-power} repeatedly,
we have the following crucial estimates \eqref{tech.lem2.eq2} and \eqref{tech.lem2.eq3}, see Section \ref{tech.lem2.pfsection} for the proof.

\begin{lem} \label{tech.lem2}
Let $1\le p<\infty, w$ be a Muckenhoupt $A_p$-weight, $A=(a(\lambda, \lambda'))_{\lambda, \lambda'\in V}\in {\mathcal B}_{r, \alpha}({\mathcal G})$ for some
$1\le r\le \infty$ and $\alpha>\tilde d_{\mathcal G}-d_{\mathcal G}/r$.
Assume that $A$ has $\ell^p_w$-stability with the optimal lower stability bound $\beta_{p, w}(A)$ and
that $1\le N\in \ZZ$ satisfies
\begin{equation}\label{tech.lem2.eq1}
\beta_{p, w} (A) \ge 2 \max(C_1, C_2 C_3) \big(A_p(w)\big)^{2/p}
\|{A}\|_{{\mathcal B}_{r,\alpha}}
\times
\left\{\begin{array}{ll} N^{-\min(1,\alpha-d_{\mathcal G}/r') } & {\rm if} \ \alpha\ne d_{\mathcal G}/r'+1\\
N^{-1} (\ln (N+1))^{1/r'}  & {\rm if} \ \alpha=d_{\mathcal G}/r'+1,
\end{array}\right.
\end{equation}
where $C_1, C_2, C_3$ are absolute constants in \eqref{tech.lem.eq1}, \eqref{tech.lem.eq2} and \eqref{norm-power} respectively.
Then there exist a matrix $H_{A,N}=(H_{A,N}(\lambda, \lambda'))_{\lambda, \lambda'\in V}$ and two absolute constants $C_4$ and $C_5$ such that
 \begin{equation} \label{tech.lem2.eq2}
 \|H_{A, N}\|_{{\mathcal B}_{r,  \alpha}}\le  C_4  N^{\alpha+d_{\mathcal G}/r}
 \end{equation}
 and
 \begin{equation}\label{tech.lem2.eq3}
|c(\lambda)|\le C_5
 (A_p(w))^{1/p} ( \beta_{p, w}(A))^{-1}
N^{d_{\mathcal G}} \sum_{\lambda'\in V}
H_{A, N}(\lambda, \lambda')|Ac(\lambda')|, \ c\in \ell^p_w.
\end{equation}
\end{lem}

Now 
we are ready to finish  the proof of Theorem \ref{mainthm}.

\begin{proof}[Proof of Theorem \ref{mainthm}]
As for $\alpha'\ge \alpha$,
 ${\mathcal B}_{r, \alpha'}({\mathcal G})$ is a Banach subalgebra of  ${\mathcal B}_{r, \alpha}({\mathcal G})$.
 Then it  suffices to prove
 \eqref{lower-bound} for all  $\alpha$ satisfying
 \begin{equation}\label{alpha.res}
  d_{\mathcal G}/r'\le \tilde d_{\mathcal G}-d_{\mathcal G}/r<\alpha\le \tilde d_{\mathcal G}-d_{\mathcal G}/r+1.
 \end{equation}
Define
\begin{equation}\label{mainthm.pf.eq1}
N_0=
\left\{\begin{array}{ll}
\tilde N_0
 & {\rm if} \ \alpha\ne d_{\mathcal G}/r'+1\\
2 \tilde N_0 (\ln (\tilde N_0+1))^{1-1/r}
 & {\rm if} \ \alpha = d_{\mathcal G}/r'+1,
\end{array}\right.
\end{equation}
where
$$\tilde N_0=\left\lfloor \Bigg (\frac{ 2 \max(C_1, C_2 C_3) \big(A_p(w)\big)^{2/p}
\|{A}\|_{{\mathcal B}_{r,\alpha}}} {\beta_{p, w}(A)} \Bigg)^{1/\min(1,\alpha-d_{\mathcal G}/r')}\right\rfloor +2.$$
Then one may verify that
\eqref{tech.lem2.eq1} is satisfied for $N=N_0$.
Applying Lemma \ref{tech.lem2} with $N$ replaced by $N_0$ and also Proposition
\ref{weighted.prop}, we have
\begin{equation}\label{mainthm.pf.eq2}
\|c\|_{q,w'}   \lesssim
( A_{q}(w'))^{1/q} ( A_p(w))^{1/p}  (\beta_{p, w}(A))^{-1}
N_{0}^{d_{\mathcal G}}
\|H_{A, N_{0}}\|_{\mathcal{B}_{1,0}}\|{A}c\|_{q,w'},  \  c\in \ell_w^p\cap \ell^{q}_{w'},
\end{equation}
where   $w'$ is  an $A_{q}$-weight with $1\le q<\infty$
and the matrix $H_{A, N}$ is given in Lemma \ref{tech.lem2}.
This together with \eqref{tech.lem2.eq2}  and  the density of $\ell_w^p\cap \ell^{q}_{w'}$ in $\ell^{q}_{w'}$
implies that
$$
\|c\|_{q, w'}\lesssim
(A_{q}(w'))^{1/q} (A_p(w))^{1/p} (\beta_{p, w}(A))^{-1}N_{0}^{\alpha+d_{\mathcal G}(1+1/r)}
\|Ac\|_{q,w'}\ {\rm for \ all} \ c\in \ell^{q}_{w'}.
$$
Therefore
\begin{equation}
\frac{
\|{A}\|_{{\mathcal B}_{r,\alpha}}} {\beta_{q, w'}(A)}
\lesssim  \frac{\big(A_{q}(w')\big)^{1/q}}{\big(A_p(w)\big)^{1/p}}
\frac{
 \big(A_{p}(w)\big)^{2/p}
\|{A}\|_{{\mathcal B}_{r,\alpha}}} {\beta_{p, w}(A)}  N_{0}^{\alpha+d_{\mathcal G}(1+1/r)},
\end{equation}
where $\beta_{q, w'}(A)$ is the optimal lower $\ell^{q}_{w'}$-stability bound of the matrix $A$.
This together with \eqref{alpha.res} and \eqref{mainthm.pf.eq1} completes the proof of Theorem \ref{mainthm}.
\end{proof}

\section{Norm-controlled inversion}
\label{normedinversion.section}

In this section, we
 show that
Banach algebras ${{\mathcal B}_{r,\alpha}}({\mathcal G})$
with $ 1 \le r \le \infty$ and $\alpha >\tilde d_{\mathcal G}-d_{\mathcal G}/r$ admit  a polynomial norm-controlled inversion in
$\mathcal B(\ell^p_w)$ for all $1\le p<\infty$ and Muckenhoupt $A_p$-weights, see Section \ref{norm-contro-inversion.thm.pfsection} for the proof.

\begin{thm}\label{norm-contro-inversion.thm}
Let $ 1 \le r \le \infty, \ 1\le p<\infty$,
${\mathcal G}:=(V,E)$ be a connected simple graph satisfying the doubling property \eqref{doubling}, and
$w$  be a Muckenhoupt $A_p$-weight.
If  $A$ belongs to ${{\mathcal B}_{r,\alpha}}({\mathcal G})$ for some $\alpha>\tilde d_{\mathcal G}-d_{\mathcal G}/r$
  and it is invertible in $\mathcal B(\ell^p_{w})$, then   $A^{-1}\in {\mathcal B}_{r, \alpha}$. Moreover, there exists an absolute constant $C$ such that
  \begin{align}\label{norm-inversion}
\|A^{-1}\|_{{\mathcal B}_{r, \alpha}}&\le C(A_{p}(w))^{1/p}\|A^{-1}\|_{\mathcal B(\ell^p_{w})}
\Big(\big(A_p(w)\big)^{2/p}\|A^{-1}\|_{\mathcal B(\ell^p_{w})} \|A\|_{{\mathcal B}_{r,\alpha}}\Big)^{(\alpha+d_{\mathcal G}(1+1/r))/\min(\alpha-d_{\mathcal G}/r', 1)}
\nonumber \\&
\quad\times
\left\{
\begin{array}{ll}
1
&{\rm if }\ \alpha\neq 1+d_{\mathcal G}/r'
\\
\Big(\ln\big((A_p(w))^{2/p}\|A^{-1}\|_{\mathcal B(\ell^p_{w})} \|A\|_{{\mathcal B}_{r,\alpha}}+1\big) \Big)^{(2 d_{\mathcal G}+1)/r'} &
{\rm if }\ \alpha= 1+d_{\mathcal G}/r',
\end{array}
\right.
\end{align}
where  $1/r'=1-1/r$.
\end{thm}

\begin{rem}\label{norm-contro-inversion.thm.remark}
{\rm
Under the assumption that the counting measure $\mu$ is Ahlfors regular in which $\tilde d_{\mathcal G}=d_{\mathcal G}$ by Proposition \ref{regular.prop}, the authors in
\cite{shinsun19} show that   Beurling algebras ${{\mathcal B}_{r,\alpha}}({\mathcal G})$ for some $1\le r\le \infty$ and $\alpha>d_{\mathcal G}(1-/r)$ admit  norm-controlled inversion  in  the symmetric $*$-algebra $\mathcal B(\ell^2)=\mathcal B(\ell^2_{w_0})$. Moreover
 for any matrix  $A\in {\mathcal B}_{r,\alpha}({\mathcal G})$ being invertible in $\mathcal B(\ell^2)$,
  \begin{align}\label{norm-inversion.old}
\|A^{-1}\|_{{\mathcal B}_{r, \alpha}}&\le C\|A^{-1}\|_{\mathcal B(\ell^2)}
\big(\|A^{-1}\|_{\mathcal B(\ell^2)} \|A\|_{{\mathcal B}_{r,\alpha}}\big)^{(\alpha+d_{\mathcal G}/r)/\min(\alpha-d_{\mathcal G}/r', 1)}
\nonumber \\&
\quad\times
\left\{
\begin{array}{ll}
1
&{\rm if }\ \alpha\neq 1+d_{\mathcal G}/r'
\\
\big(\ln\big(\|A^{-1}\|_{\mathcal B(\ell^2)} \|A\|_{{\mathcal B}_{r,\alpha}}+1\big) \big)^{(d_{\mathcal G}+1)/r'} &
{\rm if }\ \alpha= 1+d_{\mathcal G}/r',
\end{array}
\right.
\end{align}
where  $1/r'=1-1/r$ and $C$ is an absolute constant.
Therefore  under the assumption that the counting measure $\mu$ is Ahlfors regular,
the conclusion \eqref{norm-inversion.old} provides a better upper bound estimate to $\|A^{-1}\|_{{\mathcal B}_{r, \alpha}}$
than the one in \eqref{norm-inversion} with the exponent $p$ and Muckenhoupt $A_p$-weight $w$ replaced by $2$ and the trivial weight $w_0$ respectively.
}\end{rem}

 We conclude this section with a family of matrices on the lattice graph ${\mathcal Z}$
 to demonstrate the almost optimality of the  norm estimate  \eqref{norm-inversion} for the inversion.

\begin{ex}\label{exam}
{\rm
Let $A_{\kappa}:=(a_{\kappa}(n-n'))_{n,n'\in \ZZ }$, where $\kappa\in (0, 1)$ is a constant sufficiently close to one,
and
$$a_\kappa(n)= \left\{\begin{array}{ll} 1 & {\rm if} \ n=0\\
-\kappa  &  {\rm if} \ n=1\\
0 &  {\rm otherwise}.
\end{array}
\right. $$
Then its inverse $(A_\kappa)^{-1}=(\check a_{\kappa}(n-n'))_{n, n'\in \ZZ}$ is given by
$$\check a_\kappa(n)= \left\{\begin{array}{ll} \kappa^n & {\rm if} \ n\ge 0\\
0 &  {\rm otherwise}.
\end{array}
\right. $$
For $1\le r\le \infty$,  we have
\begin{equation}\label{Agamma-bound}
\|A_{\kappa}\|_{{\mathcal B}_{r,\alpha}}  
\approx 1  
\ \ {\rm and} \ \
\|A_\kappa^{-1}\|_{{\mathcal B}_{r, \alpha}} \approx (1-\kappa)^{-\alpha-1/r}.
\end{equation}

Let $w_\theta=((|n|+1)^\theta)_{n\in \ZZ}, -1<\theta<p-1$. Then
$w_\theta$ is a Muckenhoupt $A_p$-weight and
\begin{equation}\label{weighted.ex1}
\|A_\kappa^{-1}\|_{{\mathcal B}({\ell^p_{w_\theta}})}\lesssim \|A_\kappa^{-1}\|_{{\mathcal B}_{1,0}}\lesssim (1-\kappa)^{-1}.
\end{equation}
Take $c_0=(c_0(n))_{n\in \ZZ}$, where 
\begin{equation*}
c_0(n):=
\left\{
\begin{array}{ll} \kappa^n
& {\rm if } \ n\ge 0 \\
0
& {\rm otherwise. }
\end{array}
\right.
\end{equation*}
Therefore
\begin{equation} \label{weighted.ex2}
\|c_0\|_{p, w_\theta}= \Big(\sum_{n=0}^\infty \kappa^{np} (n+1)^\theta\Big)^{1/p}
\approx (1-\kappa)^{-(\theta+1)/p}
\end{equation}
and
\begin{equation} \label{weighted.ex3}
\|A_\kappa^{-1} c_0\|_{p, w}=
\Big(\sum_{n=0}^\infty \kappa^{nq} (n+1)^p (n+1)^\theta\Big)^{1/p}\approx (1-\kappa)^{-(\theta+p+1)/p}.
\end{equation}
By \eqref{weighted.ex1}, \eqref{weighted.ex2} and \eqref{weighted.ex3}, we have
\begin{equation}\label{weighted.ex4}
\|A_\kappa^{-1}\|_{{\mathcal B}(\ell^p_{w_\theta})}\approx (1-\kappa)^{-1}.
\end{equation}
Combining \eqref{Agamma-bound} and \eqref{weighted.ex4} yields
\begin{equation}
\|A_\kappa^{-1}\|_{{\mathcal B}_{r, \alpha}} \approx \|A_\kappa^{-1}\|_{\mathcal B(\ell^p_{w_\theta})}
\big(\|A_\kappa^{-1}\|_{\mathcal B(\ell^p_{w_\theta})} \|A_\kappa\|_{{\mathcal B}_{r,\alpha}}\big)^{(\alpha+d_{\mathcal G}/r-d_{\mathcal G})},
\end{equation}
while the estimate \eqref{norm-inversion} in Theorem \ref{norm-contro-inversion.thm} for $\alpha>1+d_{\mathcal G}(1-1/r)$ is
\begin{equation}
\|A_\kappa^{-1}\|_{{\mathcal B}_{r, \alpha}} \lesssim \|A_\kappa^{-1}\|_{\mathcal B(\ell^p_{w})}
\big(\|A_\kappa^{-1}\|_{\mathcal B(\ell^p_{w})} \|A_\kappa\|_{{\mathcal B}_{r,\alpha}}\big)^{(\alpha+d_{\mathcal G}/r+d_{\mathcal G})}.
\end{equation}
}\end{ex}

 \section{Proofs}\label{proofs.section}

 In this section, we collect the proofs of Propositions
 \ref{regular.prop}, \ref{beurling.prop}, \ref{beurlingonmaximalsets.pr} and \ref{weighted.prop},
 Lemmas \ref{tech.lem} and \ref{tech.lem2},
and  Theorem \ref{norm-contro-inversion.thm}.
\subsection{Proof of Proposition \ref{regular.prop}} \label{regular.prop.pfsection}
By \eqref{V1.pro}, it suffices to establish \eqref{strongpolynomial-growth}
for $N\ge 1$ and $R\ge 3$.
Let $V_N$ be a maximal $N$-disjoint set, and  define
$A_R(\lambda, N)$  as in \eqref{AR.def}. Then we obtain from \eqref{maximum2} and \eqref{ahlfors.def} that
\begin{eqnarray*}
\mu\big(A_R(\lambda, N)\big) & \le &
\frac{\sum_{\lambda_m\in A_R(\lambda, N)} \mu(B(\lambda_m, N))}
{\inf_{\lambda_{m'} \in A_R(\lambda, N)} \mu(B(\lambda_{m'}, N))}\le
B_3^{-1} N^{-d_0} \mu \big(\cup_{\lambda_m\in A_R(\lambda, N)} B(\lambda_m, N)\big)\nonumber\\
& \le &  B_3^{-1} N^{-d_0} \mu\big(B(\lambda, N+ (N+1)R\big)
\le 2^{d_0} B_3^{-1}  B_4 (R+1)^{d_0},
\end{eqnarray*}
which implies that
\begin{equation}\label{regular.prop.eq1}
\tilde d_{\mathcal G}\le d_0.
\end{equation}
Similarly by \eqref{overlap-counting} and \eqref{ahlfors.def}, we get
\begin{eqnarray*}
\mu\big(A_R(\lambda, N)\big) & \ge &
\frac{\sum_{\lambda_m\in A_R(\lambda, N)} \mu(B(\lambda_m, 2N))}
{\max_{\lambda_m \in A_R(\lambda, N)} \mu(B(\lambda_m, 2N))}\ge
B_4^{-1} (2N+1)^{-d_0} \mu \big(\cup_{\lambda_m\in A_R(\lambda, N)} B(\lambda_m, 2N)\big)\nonumber\\
& \ge &  3^{-d_0}B_4^{-1} N^{-d_0} \mu(B(\lambda, N (R-2) )
\ge 2^{-2d_0}3^{-d_0}B_4^{-1}  B_3  (R+1)^{d_0}, R\ge 3,
\end{eqnarray*}
where the third inequality holds as $B(\lambda_m, 2N))\cap B(\lambda, N(R-2))=\emptyset$ for all $\lambda_m\not\in A_R(\lambda, N)$.
This show that
\begin{equation} \label{regular.prop.eq2}
\tilde d_{\mathcal G}\ge d_0.
\end{equation}
Combining \eqref{regular.prop.eq1} and \eqref{regular.prop.eq2} 
  completes the proof.

\subsection{Proof of Proposition \ref{beurling.prop}}\label{beurling.prop.pfsection}
 The conclusion (i) is obvious and the conclusions in (ii), (iii) and (iv) are presented in
\cite[Propositions 3.3 and 3.4]{shinsun19}.
 Now we prove the conclusion (v).
Write $A=(a(\lambda, \lambda'))_{\lambda, \lambda'\in V}$ and set $h_A(n)=\sup_{\rho(\lambda, \lambda')\ge n} |a(\lambda,  \lambda')|, n\ge 0$. Then for $K\ge 1$  and $1<r\le \infty$,  we have

\begin{eqnarray*}
\|A-A_K\|_{{\mathcal B}_{1, 0}} &
\le &  2 \sum_{n=\lceil (K+1)/2\rceil}^\infty h_A(n)(n+1)^{d_{\mathcal G}-1}\le   2 \|A\|_{{\mathcal B}_{r,\alpha}}  \Big(\sum_{n=\lceil (K+1)/2\rceil}^\infty (n+1)^{-\alpha r'+d_{\mathcal G}-1}\Big)^{1/r'}
\nonumber\\
& \le &
 2 \|A\|_{{\mathcal B}_{r, \alpha}}
\Big(\int_{\lceil (K+1)/2\rceil}^\infty x^{d_{\mathcal G}-1-\alpha r'} dx
\Big)^{1/r'}
\le
\frac{2^{\alpha -d_{\mathcal G}/r'+1}}{(\alpha r' -d_{\mathcal G})^{1/r'}} \|A\|_{{\mathcal B}_{r,\alpha}}
K^{-\alpha +d_{\mathcal G}/r'},
\end{eqnarray*}
where   $r'=r/(r-1)$. This proves \eqref{AN-appr} for $1<r\le \infty$. Similarly we can prove
\eqref{AN-appr} for $r=1$.

\subsection{Proof of Proposition \ref{beurlingonmaximalsets.pr}}\label{beurlingonmaximalsets.pr.pfsection}
  The conclusion (i) follows from the definition of Beurling algebras on the graph ${\mathcal G}$ and on its
  maximal disjoint set $V_N$.

  Now we prove the conclusion (ii).  Set $\check\alpha= \alpha+ (\tilde d_{\mathcal G}-d_{\mathcal G})/r$.
   For $1\le r<\infty$, we obtain
    \begin{eqnarray}\label{Bra.eqN1}
\|A\|_{{\mathcal B}_{r, \check\alpha}}^r
& \le  & \sum_{n'=0}^\infty \sum_{n=n'(N+1)}^{(n'+1)(N+1)-1} (n+1)^{\alpha r +\tilde d_{\mathcal G}-1} \Big(\sup_{\rho(\lambda, \lambda')\ge n}
\sum_{\lambda_m\in B(\lambda, 2N), \lambda_k\in B(\lambda', 4N)} |b(\lambda_m, \lambda_k)|\Big)^r\nonumber\\
& \le &  (N+1)^{\alpha r+\tilde d_{\mathcal G}} \sum_{n'=0}^\infty  (n'+1)^{\alpha r+\tilde d_{\mathcal G}-1}
\Big(\sup_{\rho(\lambda, \lambda')\ge  n'(N+1)}
\sum_{\lambda_m\in B(\lambda, 2N), \lambda_k\in B(\lambda', 4N)} |b(\lambda_m, \lambda_k)|\Big)^r\nonumber\\
& \le & (D(\mu))^{7r}
 (N+1)^{\alpha r+\tilde d_{\mathcal G}} \sum_{n'=0}^\infty  (n'+1)^{ \alpha r+\tilde d_{\mathcal G}-1}
\Big(\sup_{\rho(\lambda_m, \lambda_k)\ge \max(n'-6, 0)(N+1)}
 |b(\lambda_m, \lambda_k)|\Big)^r\nonumber\\
 & \le & 8^{\alpha r+\tilde d_{\mathcal G}} (D(\mu))^{7r} N^{\alpha r+\tilde d_{\mathcal G}} \|B\|_{{\mathcal B}_{r,  \alpha; N}}^r,
    \end{eqnarray}
    where the third inequality follows from \eqref{overlap-counting}.
This    proves \eqref{beurlingonmaximalsets.pr.eq3} and the conclusion (ii) for $1\le r<\infty$.

 Similarly for $r=\infty$,  we have
    \begin{eqnarray}\label{Bra.eqN2}
\|A\|_{{\mathcal B}_{\infty,  \check\alpha}}& \le & \sup_{\lambda, \lambda'\in V}
 \sum_{\lambda_m\in B(\lambda, 2N), \lambda_k\in B(\lambda', 4N)} |b(\lambda_m, \lambda_k)|
(\rho(\lambda, \lambda')+1)^{\alpha} \nonumber\\
& \le & \|B\|_{\infty, \alpha; N}
\sup_{\lambda, \lambda'\in V}
 \sum_{\lambda_m\in B(\lambda, 2N),\lambda_k\in B(\lambda', 4N)} (\lfloor\rho(\lambda_m, \lambda_k)/N\rfloor+1)^{-\alpha}
(\rho(\lambda, \lambda')+1)^{ \alpha}\nonumber\\
& \le  &   8^\alpha   (D(\mu))^7 N^{\alpha} \|B\|_{\infty,  \alpha; N}.
    \end{eqnarray}
Combining \eqref{Bra.eqN1} and \eqref{Bra.eqN2} proves \eqref{beurlingonmaximalsets.pr.eq3} and the conclusion (ii).

Finally we prove the conclusion (iii).
Set $\tilde \alpha= \alpha- (\tilde d_{\mathcal G}-d_{\mathcal G})/r$ and $1/r'=1-1/r$.
Then for $1< r< \infty$, we have
\begin{eqnarray}\label{beurlingonmaximalsets.pr.pf.eq5}
\|S_{A, N}\|_{\mathcal{B}_{r,\tilde \alpha;N}} &\le &
 N^{d_{\mathcal G}  } \Big( \sum_{n=13}^\infty \big(h_A(n(N+1)/2)\big)^r (n+1)^{\alpha r + d_{\mathcal G}-1}\Big)^{1/r}\nonumber\\
 & & +
 N^{-1} \Big(\sum_{n=0}^{2N} h_A(n) (n+1)^{d_{\mathcal G}}\Big) \times \Big(\sum_{n=0}^{12} (n+1)^{ \alpha r+ d_{\mathcal G}-1}\Big)^{1/r}
\nonumber \\ & \le & 2^{2\alpha +(2 d_{\mathcal G}-1)/r} N^{-\alpha +d_{\mathcal G}/r'}
\Big(\sum_{m=4N}^\infty  ( h_A(m))^{r} (m+1)^{\alpha r+d_{\mathcal G}-1}\Big)^{1/r}\nonumber\\
 & &  +13^{\alpha +d_{\mathcal G}/r} N^{-1} \sum_{n=0}^{2N} h_A(n) (n+1)^{d_{\mathcal G}} \nonumber\\
 & \le & 2^{2\alpha +(2 d_{\mathcal G}-1)/r} N^{-\alpha +d_{\mathcal G}/r'}
 \|A\|_{{\mathcal B}_{r, \alpha}} \nonumber\\
 & &
  + 13^{\alpha +d_{\mathcal G}/r}\|{ A}\|_{{\mathcal B}_{r,\alpha}} N^{-1} \Big(\sum_{n=0}^{2N} (n+1)^{-(\alpha-1) r'+ d_{\mathcal G}-1}\Big)^{1/r'} \nonumber\\
 & \le &  2^{2\alpha +(2 d_{\mathcal G}-1)/r} N^{-\alpha +d_{\mathcal G}/r'}\|{A}\|_{{\mathcal B}_{r,\alpha}}
 +
13^{\alpha +d_{\mathcal G}/r} \|{A}\|_{{\mathcal B}_{r,\alpha}}
\nonumber\\
& &  \times \left\{\begin{array}{ll}
 2 \Big(\frac{1+|\alpha-1-d_{\mathcal G}/r'|}{|\alpha-1-d_{\mathcal G}/r'|}\Big)^{1/r'} N^{-\min(1,\alpha-d_{\mathcal G}/r')} & {\rm if} \ \alpha\ne d_{\mathcal G}/r'+1\\
 3^{1/r'} N^{-1} (\ln (N+1))^{1/r'}  & {\rm if} \ \alpha= d_{\mathcal G}/r'+1.
\end{array}\right.  
\end{eqnarray}
This proves  \eqref{SAN.estimate} for $1<r<\infty$.
Using similar argument, we can prove \eqref{SAN.estimate} for $r=1, \infty$.

\subsection{Proof of Proposition \ref{weighted.prop}}\label{weighted.prop.pfsection}
 Write $ A=(a(\lambda, \lambda'))_{\lambda, \lambda'\in V}$, 
 and set
   $h_A(n)=\sup_{\rho(\lambda, \lambda')\ge n} |a(\lambda, \lambda')|$, $n\ge 0$.
Then for any $c\in \ell^p_w$ with $1<p<\infty$, we have
\begin{eqnarray}\label{h-sum2}
\|Ac \|_{p,w} &\le &
\Big( \sum_{\lambda \in V} \Big(
\sum_{\lambda' \in V} h_A(\rho(\lambda, \lambda'))|c(\lambda')|\Big)^p w(\lambda)\Big)^{1/p}
\nonumber \\
&
\le &
 h_A(0) \|c\|_{p,w}
+\Bigg(\sum_{\lambda \in V}
\Big( \sum_{l=1}^\infty h_A(2^{l-1})  \sum_{
 2^{l-1} \le \rho(\lambda, \lambda')<2^l}
|c(\lambda')|\Big)^p w(\lambda) \Bigg)^{1/p}
\nonumber \\
&\le &  h_A(0) \|c\|_{p,w}+ \Big( \sum_{l=1}^\infty h_A(2^{l-1})2^{ld_{\mathcal G}}\Big)^{1-1/p}\nonumber\\
& & \times
\Bigg(\sum_{l=1}^\infty h_A(2^{l-1})2^{-(p-1)l d_{\mathcal G} }
\sum_{\lambda \in V}  \Big(\sum_{
 2^{l-1} \le \rho(\lambda, \lambda')<2^l}
|c(\lambda')|\Big)^{p} w(\lambda) \Bigg)^{1/p}.
\end{eqnarray}
By the equivalent definition \eqref{Aq-char} 
 of the Muckenhoupt  $A_p$-weight $w$ and the polynomial property \eqref{polynomial-growth} of the counting measure $\mu$,
we obtain
\begin{eqnarray*}\label{h-sum2+}
 & &
\sum_{\lambda \in V}  \Big(\sum_{
 \rho(\lambda, \lambda')<2^l}
|c(\lambda')|\Big)^{p} w(\lambda) \nonumber\\
 & \le &
 \sum_{\lambda \in V}
\Big(\sum_{ 
\rho(\lambda, \lambda')<2^l}
|c(\lambda')|^p  w(\lambda')\Big) \Big(\sum_{
 \rho(\lambda, \lambda^{\prime\prime})<2^l} ( w(\lambda^{\prime\prime}))^{-1/(p-1)}\Big)^{p-1}  w(\lambda)
\nonumber\\
&\le & A_p(w)
\sum_{\lambda\in V}  w(\lambda)
\Big(  \sum_{\rho(\lambda', \lambda)<2^l}
|c(\lambda')|^p  w(\lambda')\Big)
\times
 \frac{  \big(\mu(B(\lambda, 2^{l+1}-2))\big)^p} {\sum_{\rho(\lambda, \lambda^{\prime\prime})\le 2^{l+1}-2} w(\lambda^{\prime\prime})}
\nonumber\\
&\le  &   A_p(w) (D_{\mathcal G})^p 2^{p(l+1) d_{\mathcal G} }
\sum_{\lambda'\in V}
|c(\lambda')|^p  w(\lambda') \times  \sum_{\rho(\lambda, \lambda')<2^l}
 \frac{ w(\lambda)} {\sum_{\rho(\lambda', \lambda^{\prime\prime})\le 2^{l}-1} w(\lambda^{\prime\prime})}\nonumber\\
 & = &   A_p(w) (D_{\mathcal G})^p 2^{p(l+1) d_{\mathcal G} } \|c\|_{p, w}^p.
 \end{eqnarray*}
This together with \eqref{h-sum2}  
and the following estimate
\begin{equation*}\label{hA.eq00}
h_A(0)+\sum_{l=1}^\infty h(2^{l-1}) 2^{ld_{\mathcal G}}\le  h_A(0)+ 2^{2d_{\mathcal G}} \sum_{l=1}^\infty \sum_{2^{l-2}<n\le 2^{l-1}} h_A(n) (n+1)^{d_{\mathcal G}-1}\le  2^{2d_{\mathcal G}} \|A\|_{{\mathcal B}_{1, 0}}
\end{equation*}
proves \eqref{continuity} for $1<p<\infty$.


Applying a similar argument as above,
we can verify \eqref{continuity} for $p=1$.

\subsection{Proof of Lemma \ref{tech.lem}}\label{tech.lem.pfsection}
We follow the procedure used in \cite{shinsun19}, where a similar result is established for the unweighted case.
Take $\lambda_m\in V_N$. Denote the commutator between
the smooth truncation operator $\Psi_{\lambda_m}^{2N}$ and the matrix ${A}$ by  $[\Psi_{\lambda_m}^{2N}, {A}]:=\Psi_{\lambda_m}^{2N}{A}- {A} \Psi_{\lambda_m}^{2N}$, and set
$\Phi^{2N}:=\big(\sum_{\lambda_k\in V_N}\Psi_{\lambda_k}^{2N}\big)^{-1}$.
  Replacing $c$ in \eqref{mainthm.eq-1}  by $\Psi_{\lambda_m}^{2N} c$
 and applying Proposition \ref{weighted.prop}, we have
 \begin{eqnarray}\label{lowerbound1}
\beta_{p, w}(A) \|\Psi_{\lambda_m}^{2N} c\|_{p,w}
 & \le & \| A\Psi_{\lambda_m}^{2N} c\|_{p,w}
\le \|\Psi_{\lambda_m}^{2N} A c\|_{p,w}
+
\|[\Psi_{\lambda_m}^{2N}, A] c\|_{p,w}\nonumber\\
& \le &  \|\Psi_{\lambda_m}^{2N} A c\|_{p,w}+
 \|\chi_{\lambda_m}^{4N}
[\Psi_{\lambda_m}^{2N}, {A}]c \|_{p,w}
+ \|(I-\chi_{\lambda_m}^{4N}) {A} \chi_{\lambda_m}^{3N}
\Psi_{\lambda_m}^{2N}c\|_{p,w}
\nonumber \\
& \le & \|\Psi_{\lambda_m}^{2N} A c\|_{p,w}+
 \sum_{\lambda_k \in V_{N}}
\|\chi_{\lambda_m}^{4N}[\Psi_{\lambda_m}^{2N}, {A}] \chi_{\lambda_k}^{3N}
\Phi^{2N}\Psi_{\lambda_k}^{2N}c\|_{p,w}\nonumber\\
 & & + 2^{3d_{\mathcal G}} D_{\mathcal G}  \big(A_p(w)\big)^{1/p}
 \|(I-\chi_{\lambda_m}^{4N}) {A} \chi_{\lambda_m}^{3N}\|_{{\mathcal B}_{1, 0}}
 \|\Psi_{\lambda_m}^{2N}c\|_{p,w},
\end{eqnarray}
where   the second inequality holds, as
$(I-\chi_{\lambda_m}^{4N}) [\Psi_{\lambda_m}^{2N}, {A}]=
(I-\chi_{\lambda_m}^{4N}) A \Psi_{\lambda_m}^{2N}= (I-\chi_{\lambda_m}^{4N}) A \chi_{\lambda_m}^{3N}\Psi_{\lambda_m}^{2N}$
by the supporting properties for $\chi_{\lambda_m}^{3N}, \chi_{\lambda_m}^{4N}$ and  $\Psi_{\lambda_m}^{2N}$.
From \eqref{solid.eq} and  \eqref{AN-appr} in  Proposition \ref{beurling.prop}, we obtain
\begin{equation}\label{lowerbound1+}
\|(I-\chi_{\lambda_m}^{4N}) {A} \chi_{\lambda_m}^{3N}\|_{{\mathcal B}_{1, 0}}
\le \|A-{A}_{N}\|_{{\mathcal B}_{1, 0}}\le C_0 \|A\|_{{\mathcal B}_{r,\alpha}} N^{-\alpha  +d_{\mathcal G}(1-1/r)}.
\end{equation}
 Combining
\eqref{lowerbound1} and \eqref{lowerbound1+} yields
 \begin{eqnarray*} 
\beta_{p, w}(A)  \|\Psi_{\lambda_m}^{2N} c\|_{p,w}
& \le & \|\Psi_{\lambda_m}^{2N} A c\|_{p,w}+
 \sum_{\lambda_k \in V_{N}}
\|\chi_{\lambda_m}^{4N}[\Psi_{\lambda_m}^{2N}, {A}] \chi_{\lambda_k}^{3N}
\Phi^{2N}\Psi_{\lambda_k}^{2N}c\|_{p,w}\nonumber\\
 & & + 2^{3d_{\mathcal G}}  C_0 D_{\mathcal G}   (A_p(w))^{1/p}
 \|A\|_{{\mathcal B}_{r,\alpha}} N^{-\alpha  +d_{\mathcal G}(1-1/r)}
 \|\Psi_{\lambda_m}^{2N}c\|_{p,w}.
\end{eqnarray*}
This together with \eqref{tech.lem.eq1} proves that
 \begin{equation} \label{lowerbound1++}
\beta_{p, w}(A)  \|\Psi_{\lambda_m}^{2N} c\|_{p,w}
\le 2 \|\Psi_{\lambda_m}^{2N} A c\|_{p,w}+
2 \sum_{\lambda_k \in V_{N}}
\|\chi_{\lambda_m}^{4N}[\Psi_{\lambda_m}^{2N}, {A}] \chi_{\lambda_k}^{3N}
\Phi^{2N}\Psi_{\lambda_k}^{2N}c\|_{p,w}.
\end{equation}

For $\lambda_k\in V_{N}$ with
$\rho(\lambda_m, \lambda_k)> 12(N+1)$, we obtain from  the finite covering property \eqref {overlap-counting} for the maximal $N$-disjoint set $V_N$,
the  equivalent definition \eqref{Aq-char} of the weight $w$,
the polynomial growth property  \eqref{polynomial-growth} of the counting measure $\mu$, and the monotonicity of $h_A(n), n\ge 0$,
  that
\begin{eqnarray}\label{lowerbound4}
 & & \|\chi_{\lambda_m}^{4N } [\Psi_{\lambda_m}^{2N}, {A}]\chi_{\lambda_k}^{3N}\Phi^{2N}
\Psi_{\lambda_k}^{2N} c\|_{p,w}
 =  \|\Psi_{\lambda_m}^{2N} {A}\chi_{\lambda_k}^{3N}\Phi^{2N}\Psi_{\lambda_k}^{2N} c\|_{p,w}
\nonumber \\&
\le &
h_{ A}\Big(\frac{\rho(\lambda_m, \lambda_k)}{2}\Big)
\Bigg(\sum_{\lambda\in B(\lambda_m, 4N)}
\Big(\sum_{\lambda'\in B(\lambda_k, 4N)}
|\Psi_{\lambda_k}^{2N} c(\lambda')|\Big)^p w(\lambda)
\Bigg)^{1/p}
\nonumber \\&
\lesssim & 
(A_p(w))^{1/p} S_{A, N}(\lambda_m, \lambda_k)
\|\Psi_{\lambda_k}^{2N}c \|_{p,w}
\Bigg( \frac{w\big(B(\lambda_m, 4N)\big)}{w\big(B(\lambda_k, 4N)\big)}\Bigg)^{1/p}.
\end{eqnarray}

Set
$\tilde A_M=\big(|a(\lambda, \lambda')| \rho(\lambda, \lambda')\chi_{[0, M]}(\rho(\lambda, \lambda'))\big)_{\lambda, \lambda'\in V},\ M\ge 0$.
 For  $\lambda_k \in V_{N}$ with $\rho(\lambda_m, \lambda_k)< 12(N+1)$,
we have
\begin{eqnarray}\label{8N-bound}
 \|\chi_{\lambda_m}^{4N}
[\Psi_{\lambda_m}^{2N}, {A}]\chi_{\lambda_k}^{3N}\Phi^{2N}\Psi_{\lambda_k}^{2N} c\|_{p,w}
& \lesssim  &  
\big( A_p(w)\big)^{1/p}
\|\chi_{\lambda_m}^{4N}[\Psi_{\lambda_m}^{2N}, {A}]\chi_{\lambda_k}^{3N}\|_{{\mathcal B}_{1, 0}}
\|\Phi^{2N}\Psi_{\lambda_k}^{2N} c\|_{p,w}\nonumber\\
&\lesssim&
 \big(A_p(w)\big)^{1/p}  N^{-1}
\|\tilde A_{19N+12}\|_{{\mathcal B}_{1, 0}}
\|\Psi_{\lambda_k}^{2N} c\|_{p,w}, 
\end{eqnarray}
where the first inequality
 follows from  the weighted norm inequality \eqref{continuity} in
Proposition \ref{weighted.prop},  and the second one holds by
 the solidness of the Banach algebra ${\mathcal B}_{1, 0}({\mathcal G})$ in Proposition \ref{beurling.prop} and the Lipschitz property for the  trapezoid function $\psi_0$. 
 Observe that  
 \begin{equation}\label{8N-bound+1}
 w\big(B(\lambda_k, 4N)\big)\le w\big(B(\lambda_m, 19N+12)\big)\le  (D(\mu))^{3p}  A_p(w)  w\big(B(\lambda_m, 4N)\big)
  \end{equation}
    by the double property \eqref{weighteddoubling} for the $A_p$-weight $w$, and
\begin{eqnarray} \label{8N-bound+2} \|\tilde A_{19N+12}\|_{{\mathcal B}_{1, 0}}
& = & \sum_{n=0}^{19N+12}  (n+1)^{d_{\mathcal G}-1} \sup_{n\le \rho(\lambda, \lambda')\le 19N+12} |a(\lambda, \lambda')| \rho(\lambda, \lambda')\nonumber\\
 & \le & 2 \sum_{n=0}^{19N+12} (n+1)^{d_{\mathcal G}-1} \sum_{n/2\le m\le 19N+12} h_{A} (m)
\lesssim  
 \sum_{n=0}^{2N} h_A(n) (n+1)^{d_{\mathcal G}}
\end{eqnarray}
by the monotonicity of $h_A(n), n\ge 0$.
Combining \eqref{8N-bound}, \eqref{8N-bound+1}  and \eqref{8N-bound+2}, we get
\begin{equation}\label{8N-bound+3}
 \|\chi_{\lambda_m}^{4N}
[\Psi_{\lambda_m}^{2N}, {A}]\chi_{\lambda_k}^{3N}\Phi^{2N}\Psi_{\lambda_k}^{2N} c\|_{p,w}
  \lesssim  \big(A_p(w)\big)^{2/p}  
S_{A, N}(\lambda_m, \lambda_k)
 \|\Psi_{\lambda_k}^{2N}c \|_{p,w}
\Bigg( \frac{w\big(B(\lambda_m, 4N)\big)}{w\big(B(\lambda_k, 4N)\big)}\Bigg)^{1/p}
\end{equation}
if $\rho(\lambda_m, \lambda_k)\le 12 (N+1)$.

Combining  \eqref{lowerbound1++},
\eqref{lowerbound4} and \eqref{8N-bound+3} proves \eqref{tech.lem.eq2}.

\subsection{Proof of Lemma \ref{tech.lem2}}\label{tech.lem2.pfsection}
Set
$\alpha_{\lambda_m}:= w\big(B(\lambda_m, 4N)\big), \lambda_m\in V_N$,
and write
$$\big(S_{A, N}\big)^l:=\big(S_{A, N; l}(\lambda_m, \lambda_k)\big)_{\lambda_m, \lambda_k\in V_N}, l\ge 1.$$
By \eqref{tech.lem2.eq1},  the integer $N$ satisfies
\eqref{tech.lem.eq1} and hence \eqref{tech.lem.eq2} holds by Lemma \ref{tech.lem}.
Applying \eqref{tech.lem.eq2} repeatedly, we get
\begin{eqnarray}\label{tech.lem.eq2++}
 & &   \big(\alpha_{\lambda_m}\big)^{-1/p}
\|\Psi_{\lambda_m}^{2N} c\|_{p,w} \nonumber\\
& \le & 2 (\beta_{p, w}(A))^{-1} \big(\alpha_{\lambda_m}\big)^{-1/p}
\|\Psi_{\lambda_m}^{2N} {A}c\|_{p,w} \nonumber\\
& &
+ C_2 \big(A_p(w)\big)^{2/p} (\beta_{p, w}(A))^{-1}
\sum_{\lambda_k\in V_N} S_{A, N}(\lambda_m, \lambda_k)  \big(\alpha_{\lambda_k}\big)^{-1/p} \|\Psi_{\lambda_k}^{2N} c\|_{p,w}\nonumber\\
& \le & \cdots\nonumber\\
& \le &  2 (\beta_{p, w}(A))^{-1} \big(\alpha_{\lambda_m}\big)^{-1/p}
\|\Psi_{\lambda_m}^{2N} {A}c\|_{p,w}+ 2 (\beta_{p, w}(A))^{-1}\nonumber\\
& & \times  \sum_{l=1}^{L-1}
\big(C_2 \big(A_p(w)\big)^{2/p} (\beta_{p, w}(A))^{-1}\big)^l
\sum_{\lambda_k\in V_N} S_{A, N; l}(\lambda_m, \lambda_k)  \big(\alpha_{\lambda_k}\big)^{-1/p} \|\Psi_{\lambda_k}^{2N} Ac\|_{p,w}\nonumber\\
 & & +  (C_2 \big(A_p(w)\big)^{2/p} (\beta_{p, w}(A))^{-1}\big)^L
\sum_{\lambda_k\in V_N} S_{A, N; L}(\lambda_m, \lambda_k)  \big(\alpha_{\lambda_k}\big)^{-1/p} \|\Psi_{\lambda_k}^{2N} c\|_{p,w},
\end{eqnarray}
where $L\ge 2$.
Define
\begin{equation}\label{wan.def0}
W_{A, N}= 2 I+2 \sum_{l=1}^\infty \big(C_2 \big(A_p(w)\big)^{2/p} (\beta_{p, w}(A))^{-1}\big)^l (S_{A, N})^l.
\end{equation}
Then by \eqref{norm-power} and \eqref{tech.lem2.eq1}, we have
\begin{eqnarray}\label{wan.estimate++}
\|W_{A, N}\|_{{\mathcal B}_{r, \alpha- (\tilde d_{\mathcal G}-d_{\mathcal G})/r; N}} & \le &
2+2\sum_{l=1}^\infty \big(C_2  C_3 \big(A_p(w)\big)^{2/p} (\beta_{p, w}(A))^{-1} \|{A}\|_{{\mathcal B}_{r,\alpha}}\big)^l\nonumber\\
 & & \quad \times \left\{\begin{array}{ll} N^{-\min(1,\alpha-d_{\mathcal G}/r')l } & {\rm if} \ \alpha\ne d_{\mathcal G}/r'+1\\
N^{-l} (\ln (N+1))^{l/r'}  & {\rm if} \ \alpha= d_{\mathcal G}/r'+1,
\end{array}\right.   \nonumber\\
&  \le &  2+2\sum_{l=1}^\infty 2^{-l} =4.\end{eqnarray}

Following the argument used in the proof of Proposition \ref{weighted.prop}, we obtain
\begin{eqnarray}
 & & \Big( \sum_{\lambda_m\in V_N} \Big|
\sum_{\lambda_k\in V_N} S_{A, N; L}(\lambda_m, \lambda_k)  \big(\alpha_{\lambda_k}\big)^{-1/p} \|\Psi_{\lambda_k}^{2N} c\|_{p,w}\Big|^p
\alpha_{\lambda_m}\Big)^{1/p}\nonumber\\
& \lesssim  & \big \| (S_{A, N})^L\big\|_{{\mathcal B}_{1, 0; N}}
\Big(\sum_{\lambda_k\in V_N} \Big|\big(\alpha_{\lambda_k}\big)^{-1/p} \|\Psi_{\lambda_k}^{2N} c\|_{p,w}\Big|^p \alpha_{\lambda_k}\Big)^{1/p}\nonumber\\
& \le &  C_6 \big\| (S_{A, N})^L\big\|_{{\mathcal B}_{r, \alpha- (\tilde d_{\mathcal G}-d_{\mathcal G})/r; N}} \|c\|_{p, w},
\end{eqnarray}
where $C_6$ is an absolute constant.
This together with
\eqref{norm-power} and \eqref{tech.lem2.eq1}
implies that
\begin{eqnarray}\label{wan.estimate++2}
& & \Big(\sum_{\lambda_m\in V_N} \Big|
\sum_{\lambda_k\in V_N} S_{A, N; L}(\lambda_m, \lambda_k)  \big(\alpha_{\lambda_k}\big)^{-1/p} \|\Psi_{\lambda_k}^{2N} c\|_{p,w}\Big|^p
\alpha_{\lambda_m}\Big)^{1/p}\nonumber\\
& &\quad \times \Big(C_2 \big(A_p(w)\big)^{2/p} (\beta_{p, w}(A))^{-1}\Big)^L
\le  C_6 2^{-L} \|c\|_{p, w}\to 0\ {\rm as} \ L\to \infty.
\end{eqnarray}
Taking limit $L\to \infty$ in \eqref{tech.lem.eq2++} and applying \eqref{wan.estimate++} and \eqref{wan.estimate++2}, we obtain
\begin{eqnarray} \label{wan.estimate++3}
 & &  \beta_{p, w}(A) \big(\alpha_{\lambda_m}\big)^{-1/p}
\|\Psi_{\lambda_m}^{2N} c\|_{p,w} \le \sum_{\lambda_k\in V_N} W_{A,N}(\lambda_m, \lambda_k)
\big(\alpha_{\lambda_k}\big)^{-1/p}
\|\Psi_{\lambda_k}^{2N} Ac\|_{p,w}, \ \lambda_m\in V_N,
\end{eqnarray}
where $c\in \ell^p_w$.

Define
$H_{A, N}:=\big(H_{A, N}(\lambda, \lambda')\big)_{\lambda, \lambda'\in V}$ by
\begin{equation}\label{def-H}
H_{A, N}(\lambda, \lambda'):=\sum_{\lambda_m\in B(\lambda, 2N)}
\sum_{\lambda_k \in B(\lambda', 4N)}
W_{{A}, N}(\lambda_m, \lambda_k ).
\end{equation}
Then the desired norm estimate \eqref{tech.lem2.eq2} 
 follows from \eqref{solid.eq},
\eqref{beurlingonmaximalsets.pr.eq3} and \eqref{wan.estimate++}.

Let  $\lambda\in V$ and select
 $\lambda_m\in V_{N}$ such that $\lambda\in B(\lambda_m, 2N)$. Such a fusion vertex
$\lambda_m$ exists by the covering property \eqref{overlap-counting}.
Replacing the vector $(c(\lambda'))_{\lambda'\in V}$ and the ball $B$ by the delta vector $(\delta_0(\lambda, \lambda'))_{\lambda'\in V}$ and $B(\lambda_m, 4N)$
in \eqref{Aq-char}, respectively, we  get
\begin{equation}\label{est-c}
\alpha_{\lambda_m}  
\lesssim  A_p(w)  N^{pd_{\mathcal G}}
w(\lambda).
\end{equation}
Combining  \eqref{tech.lem.eq2}, \eqref{wan.estimate++3} and \eqref{est-c}, we obtain
\begin{eqnarray}\label{preresult}
|c(\lambda)|&\lesssim &
( A_p(w))^{1/p} N^{d_{\mathcal G}} \sum_{\lambda_m\in B(\lambda, 2N)} \alpha_{\lambda_m}^{-1/p}
\|\Psi_{\lambda_m}^{2N}c \|_{p,w}
\nonumber \\
& \lesssim &  ( A_p(w))^{1/p} (\beta_{p, w}(A))^{-1}
 N^{d_{\mathcal G}} \sum_{\lambda_m\in B(\lambda, 2N)} \sum_{\lambda_k \in V_{N}}
W_{A, N}(\lambda_m, \lambda_k )\alpha_{\lambda_k}^{-1/p}
\|\Psi_{\lambda_k}^{2N} {A}c\|_{p,w}
\nonumber \\
&
\lesssim &
 (A_p(w))^{1/p}  (\beta_{p, w}(A))^{-1}
N^{d_{\mathcal G}} \sum_{\lambda'\in V}
H_{{A}, N}(\lambda, \lambda')|{A}c(\lambda')|\ \ {\rm for \ all} \ c\in \ell^p_w,
\end{eqnarray}
where the last inequality holds as
$\alpha_{\lambda_k}^{-1/p}
\|\Psi_{\lambda_k}^{2N} {A}c\|_{p,w}\le \|\Psi_{\lambda_k}^{2N} {A}c\|_{p,w_0}\le \|\Psi_{\lambda_k}^{2N} {A}c\|_{1,w_0}.$
This proves \eqref{tech.lem2.eq3}.

\subsection{Proof of Theorem \ref{norm-contro-inversion.thm}}\label{norm-contro-inversion.thm.pfsection}
By the invertibility assumption of the matrix  $A$
 in $\ell^p_w$,
it has the $\ell^p_w$-stability \eqref{mainthm.eq-1} and its optimal lower   stability bound  $\beta_{p, w}(A)$ satisfies
\begin{equation} \label{wiener.pff.eq1}
\beta_{p, w}(A) \ge \big(\|A^{-1} \|_{{\mathcal B}(\ell^p_w)}\big)^{-1}.
\end{equation}
Let $r'$ be the conjugate exponent of $r$, i.e., $1/r+1/r'=1$,   $N\ge 2$ be an integer   satisfying
\begin{equation}\label{wiener.pff.eq3}
\big(\|A^{-1} \|_{{\mathcal B}(\ell^p_w)}\big)^{-1}\ge 2 \max(C_2 C_3, C_1) \big(A_p(w)\big)^{2/p}
\|{A}\|_{{\mathcal B}_{r,\alpha}}
\times
\left\{\begin{array}{ll} N^{-\min(1,\alpha-d_{\mathcal G}/r') } & {\rm if} \ \alpha\ne d_{\mathcal G}/r'+1\\
N^{-1} (\ln (N+1))^{1/r'}  & {\rm if} \ \alpha=d_{\mathcal G}/r'+1,
\end{array}\right.
\end{equation}
and
$H_{A, N}=(H_{A, N}(\lambda, \lambda'))_{\lambda, \lambda'\in V}$ be as in \eqref{def-H} except replacing $\beta_{p, w}(A)$ by
$(\|A^{-1} \|_{{\mathcal B}(\ell^p_w)})^{-1}$.
Following the argument used in the proof of Lemma \ref{tech.lem2},
we obtain
\begin{equation}\label{wiener.pff.eq2-}
\|H_{A,N}\|_{{\mathcal B}_{r, \alpha}}\lesssim 
N^{\alpha+d_{\mathcal G}/r}
\end{equation}
and
\begin{equation}\label{wiener.pff.eq2}
|c(\lambda)|
\lesssim
 (A_p(w))^{1/p}  \|A^{-1} \|_{{\mathcal B}(\ell^p_w)}
N^{d_{\mathcal G}} \sum_{\lambda'\in V}
H_{A, N}(\lambda, \lambda')|(Ac)(\lambda')|, \ c=(c(\lambda))_{\lambda\in V}\in \ell^p_w.
\end{equation}

Write $A^{-1}:=(\check{a}(\lambda', \lambda))_{\lambda', \lambda\in V}$ and denote
$\check{a}_{\lambda}:=(\check{a}(\lambda', \lambda))_{\lambda'\in V}, \ \lambda\in V$.
Then  $\check{a}_{\lambda}\in \ell^p_w$ by \eqref{wiener.pff.eq1} and  the invertibility of the matrix $A$. 
Replacing $c$ in \eqref{wiener.pff.eq2} by $\check{a}_{\lambda}$,
we get
\begin{eqnarray}\label{wiener.pff.eq4}
|\check{a}(\lambda', \lambda)|
 & \lesssim &
 (A_p(w))^{1/p}  \|A^{-1} \|_{{\mathcal B}(\ell^p_w)}
N^{d_{\mathcal G}} \sum_{\lambda^{\prime\prime}\in V}
H_{{A}, N}(\lambda', \lambda^{\prime\prime})|({A}\check a_{\lambda}) (\lambda^{\prime\prime})|\nonumber\\
& = & (A_p(w))^{1/p}  \|A^{-1} \|_{{\mathcal B}(\ell^p_w)} N^{d_{\mathcal G}} H_{{A}, N}(\lambda', \lambda)\ \ {\rm for \ all} \ \lambda, \lambda'\in V.
\end{eqnarray}
This together with  \eqref{wiener.pff.eq2-}  and the solidness of the Beurling algebra
${\mathcal B}_{r, \alpha}({\mathcal G})$ in Proposition \ref{beurling.prop}
implies that
\begin{equation}  \label{wiener.pff.eq5}
\|A^{-1}\|_{{\mathcal B}_{r,  \alpha}} \lesssim
(A_p(w))^{1/p}  \|A^{-1} \|_{{\mathcal B}(\ell^p_w)} N^{d_{\mathcal G}} \|H_{{A}, N}\|_{{\mathcal B}_{r, \alpha}}
\lesssim  (A_p(w))^{1/p}  \|A^{-1} \|_{{\mathcal B}(\ell^p_w)}N^{\alpha+d_{\mathcal G}(1+1/r)}.
\end{equation}

Define
\begin{equation}\label{tech.lem4.N1}
N_1=
\left\{\begin{array}{ll}
\tilde N_1
 & {\rm if} \ \alpha\ne d_{\mathcal G}/r'+1\\
2 \tilde N_1 (\ln (\tilde N_1+1))^{1/r'}
 & {\rm if} \ \alpha= d_{\mathcal G}/r'+1,
\end{array}\right.
\end{equation}
where
$$\tilde N_1=\left\lfloor \Big (2 \max(C_1, C_2 C_3) \big(A_p(w)\big)^{2/p}\|A^{-1}\|_{\mathcal B(\ell^p_{w})}
\|{A}\|_{{\mathcal B}_{r,\alpha}}\Big)^{1/\min(1,\alpha-d_{\mathcal G}/r')}\right\rfloor +2$$
and  $C_1, C_2, C_3$ are absolute constants in \eqref{tech.lem.eq1}, \eqref{tech.lem.eq2} and \eqref{norm-power} respectively.
One may verify that  $N_1$ satisfies \eqref{wiener.pff.eq3}.
Then replacing $N$ in \eqref {wiener.pff.eq5} by the above integer $N_1$
completes the proof.

\end{document}